\documentclass{amsart}

\usepackage{amsfonts}
\usepackage{amsmath}
\usepackage{amssymb}
\usepackage{amsthm}
\usepackage[dvips]{graphicx}

\newtheorem{theorem}{\normalfont\scshape Theorem}
\newtheorem{proposition}[theorem]{\normalfont\scshape Proposition}

\title[A numerical method for a Klein-Gordon equation]{A numerical method for computing radially symmetric
        solutions of a dissipative nonlinear modified Klein-Gordon
        equation}
\author{J. E. Mac\'{\i}as-D\'{\i}az}
\address{Department of Physics, University of New Orleans, 
        New Orleans, LA 70148}
\email{jemacias@uno.edu}

\author{A. Puri}
\address{Department of Physics, University of New Orleans, 
        New Orleans, LA 70148}
\email{apuri@uno.edu}

\date{\today}

\begin{document}

\begin{abstract}
In this paper we develop a finite-difference scheme to approximate
radially symmetric solutions of the initial-value problem with
smooth initial conditions
\begin{equation}
\begin{array}{c}
    \displaystyle{\frac {\partial ^2 w} {\partial t ^2} - \nabla ^2 w
    - \beta \frac {\partial} {\partial t} \left( \nabla ^2 w \right)
    + \gamma \frac {\partial w} {\partial t} + m ^2 w + G ^\prime (w)
    = 0} \\
    \begin{array}{rl}
        \begin{array}{l}
            {\rm subject\ to:} \qquad \\ \\ \\
        \end{array}
        \left\{
        \begin{array}{ll}
            w (\bar{x} , 0) = \phi (\bar{x}), & \bar{x} \in D \\
            \displaystyle {\frac {\partial w} {\partial t} (\bar{x} ,
            0)} = \psi (\bar{x}), & \bar {x} \in D
        \end{array}\right.
    \end{array} \\ \\ \\
\end{array}
\label{paperproblem}
\end{equation}
in an open sphere $D$ around the origin, where the internal and
external damping coefficients---$\beta$ and $\gamma$,
respectively---are constant, and the nonlinear term has the form
$G ^\prime (w) = w ^p$, with $p > 1$ an odd number. The functions 
$\phi$ and $\psi$ are radially symmetric in $D$, and $\phi$, $\psi$,
$r \phi$ and $r \psi$ are assumed to be small at infinity. We 
prove that our scheme is consistent order $\mathcal {O} ( \Delta t ^2 ) +
\mathcal {O} ( \Delta r ^2 )$ for $G ^\prime$ identically equal to
zero, and provide a necessary condition for it to be stable order
$n$. Part of our study will be devoted to compare the physical
effects of $\beta$ and $\gamma$.
\end{abstract}

\maketitle

\section{Introduction}

Klein-Gordon-like equations appear in several branches of modern
physics. A modified sine-Gordon equation appears for instance in
the study of long Josephson junctions between superconductors when
dissipative effects are taken into account \cite{Solitons}. A
similar partial differential equation with different nonlinear
term appears in the study of fluxons in Josephson tramsmission
lines \cite{Lomdahl}. A modified Klein-Gordon equation appears in
the statistical mechanics of nonlinear coherent structures such as
solitary waves in the form of a Langevin equation (see \cite{Makhankov} 
pp. 298--309); here no internal damping coefficient appears,
though. Finally, our differential equation describes the motion of
a damped string in a non-Hookean medium.

\smallskip
The classical $(1 + 1)$-dimensional linear Klein-Gordon equation
has an exact soliton-like solution in the form of a traveling wave
\cite{RussianBook}. Some results concerning the analytic behavior
of solutions of nonlinear Klein-Gordon equations have been
established \cite{Glassey,Jorgens,Barone}; however, no exact
method of solution is known for arbitrary initial-value problems
involving this equation. From that point of view it is important
to investigate numerical techniques to describe the evolution of
radially symmetric solutions of (\ref{paperproblem}).

\smallskip
It is worth mentioning that some numerical research has been done
in this direction. Strauss and V\'{a}z\-quez \cite{StraussVazquez}
developed a finite-difference scheme to approximate radially
symmetric solutions of the nonlinear Klein-Gordon equation for the
same nonlinear term we study in this paper; one of the most
important features of their numerical method was that the discrete
energy associated with the differential equation is conserved. The
numerical study of the sine-Gordon model that describes the
Josephson tunnel junctions has been undertaken by Lomdahl {\it et
al.} \cite{Lomdahl}. Numerical simulations have also been
performed to solve the $(1 + 1)$-dimensional Langevin equation
\cite{AlexHabib}.

\smallskip
In this paper we extend Strauss and V\'{a}zquez's technique to
include the effects of both internal and external damping, and validate
our results against those in \cite{StraussVazquez}. Section
\ref{sec2} is devoted to setting up the finite-difference scheme;
the energy analysis of our problem is also carried out. Numerical
results are presented in Section \ref{sec3}, followed by a brief
discussion.

\section{Analysis}
\label{sec2}

\subsection*{Analytical results}

The following is the major theoretic result we will use in our
investigation. Here $M (t)$ represents the amplitude of a solution
of (\ref{paperproblem}) at time $t$, that is
\begin{equation}
M (t) = \max _{x} | w (x , t) | . \nonumber
\end{equation}

\begin{theorem}
Let $\beta$ and $\gamma$ be both equal to zero, and let $G ^\prime
(w) = | w | ^{p - 1} w$. Suppose that $\phi$ and $\psi$ are smooth
and small at infinity. Then
\begin{enumerate}
\setlength{\itemsep}{0pt}
\item[{\rm (1)}] If $p < 5$, a unique smooth solution of {\rm
(\ref{paperproblem})} exists with amplitude bounded at all time
{\rm \cite{Jorgens}}.

\item[{\rm (2)}] If $p \geq 5$, a weak solution exists for all
time {\rm \cite{Segal}}.

\item[{\rm (3)}] For $p > 8/3$ and for solutions of bounded
amplitude, there is a scattering theory; in particular, they decay
uniformly as fast as $M (t) \leq c (1 + |t|) ^{ - 3 / 2}$ {\rm
\cite{Morawetz}}. \qedhere \qed
\end{enumerate}
\end{theorem}

\subsection*{Finite-difference scheme}

Throughout this section we will assume that the functions $\phi (
\bar {x} )$ and $\psi ( \bar {x} )$ are smooth, of compact
support, radially symmetric in the open sphere $D$ with center
in the origin and radius $L$, and that $\phi$, $\psi$,  
$r \phi$ and $r \psi$ are small at infinity in $D$. Moreover, we 
will suppose that $w (\bar {x} , t)$ is a radially symmetric solution of
(\ref{paperproblem}).

\smallskip
Let $r = || \bar {x} ||$ be the Euclidean norm of $\bar {x}$ and let 
$G ^\prime (w) = w ^p$, for $p>1$ an odd number.
Setting $v (r , t) = r w (r , t)$ for every $0 < r < L$ and $t \in
\mathbb {R}$, it is evident that $v$ must satisfy the mixed-value
problem
\begin{equation}
\begin{array}{c}
    \displaystyle {\frac {\partial ^2 v} {\partial t ^2} - \frac
    {\partial ^2 v} {\partial r ^2} + \gamma \frac {\partial v}
    {\partial t} - \beta \frac {\partial ^3 v} {\partial t \
    \partial r ^2}+ m ^2 v + r G ^\prime ( v / r ) = 0} \\
    \begin{array}{rl}
        \begin{array}{l}
            {\rm subject\ to:} \qquad \\ \\ \\
        \end{array}
        \left\{
        \begin{array}{ll}
            v (r , 0) = r \phi (r), & 0 < r < L \\
            \displaystyle {\frac {\partial v} {\partial t} (r ,
            0)} = r \psi (r), & 0 < r < L \\
            v (0 , t) = 0, & t \geq 0.
        \end{array}\right.
    \end{array}
\end{array}
\label{KGur4-4}
\end{equation}

Proceeding now to discretize our problem, 
let $a < L$ be a positive number with the property that $\phi$ and
$\psi$ vanish outside of the sphere with center in the origin and
radius $a - \epsilon$, for some $\epsilon > 0$. Let $0 = r_0 < r
_1 < \dots < r _M = a$ and $0 = t _0 < t _1 < \dots < t _N = T$ be
partitions of $[0 , a]$ and $[0 , T]$, respectively, into $M$ and
$N$ subintervals of lengths $\Delta r = a / M$ and $\Delta t = T /
N$, respectively. Denote the approximate value of $v (r _j , t
_n)$ by $v _j ^n$. The finite-difference scheme associated with
(\ref{KGur4-4}) is
\begin{eqnarray}
\frac {v _j ^{n + 1} - 2 v _j ^n + v _j ^{n - 1}} {(\Delta t) ^2}
- \frac {v _{j + 1} ^n - 2 v _j ^n + v _{j - 1} ^n} {(\Delta r)
^2} + \gamma \frac {v _j ^{n + 1} - v _j ^{n - 1}} {2 \Delta t} \
- \qquad & & \nonumber \\
\beta \frac {\left( v _{j + 1} ^{n + 1} - 2 v _j ^{n + 1} + v _{j
- 1} ^{n + 1}\right) - \left( v _{j + 1} ^{n - 1} - 2 v _j ^{n -
1} + v _{j - 1} ^{n - 1} \right)} {2 \Delta t \left( \Delta r
\right) ^2} \ + \qquad & & \label{EasyScheme1} \\
\frac {m ^2} {2} \left[ v _j ^{n + 1} + v _j ^{n - 1} \right] +
\frac {1} {(j \Delta r) ^{p - 1}} \frac {G ( v _ j ^{n + 1}) - G (
v _j ^{n - 1})} {v _j ^{n + 1} - v _j ^{n - 1}} & = & 0, \nonumber
\end{eqnarray}
where $G (v) =  v  ^{p + 1} / (p + 1)$.

\smallskip
Computationally, our method requires an application of Newton's
method for systems of nonlinear equations along with Crout's
reduction technique for tridiogonal linear systems.

\smallskip
An interesting property of this finite-difference scheme is that,
for initial approximations $\{ w _j ^0 \}$ and $\{ w _j ^1 \}$ with
zero centered-difference first spatial derivatives at the origin, the 
successive approximations provided by the method have likewise 
centered-difference first spatial derivative equal to zero at the origin. 
This claim follows by induction using the facts that
\begin{equation*}
v _1 ^k + v _{- 1} ^k + \beta \frac {(v _1 ^{k + 1} + v _{- 1} ^{k + 1})
- ( v _1 ^{k - 1} + v _{- 1} ^{k - 1} )} {2 \Delta t} = 0,
\end{equation*}
that $w _1 ^n = w _{- 1} ^n$ if{}f $v _1 ^n + v _{- 1} ^n = 0$, 
and the substitution $v _j ^n = j w _j ^n \Delta r$. The induction 
hypothesis implies that $v _1 ^{n} + v _{- 1} ^{n} = 
0$ for every $n \leq k$, whence the claim follows.

\smallskip
Our last statement implies that for a smooth initial profile at the 
origin, the subsequent approximations yielded by our method will be
likewise smooth. As a test case, it is worthwhile mentioning that we have successfully
obtained numerical results to verify this 
claim using a Gaussian initial profile centered at the origin.

\subsection*{Stability analysis}

It is clear that (\ref{EasyScheme1}) is consistent order $\mathcal
{O} (\Delta t ^2) + \mathcal {O} (\Delta r ^2)$ with
(\ref{KGur4-4}) whenever $G ^\prime$ is identically equal to zero.
Moreover, in order for the finite-difference scheme to be stable
order $n$ it is necessary that
\begin{eqnarray}
\left( \frac {\Delta t} {\Delta r} \right) ^2 & < & 1 + \gamma
\frac {\Delta t} {4} + \beta \frac {\Delta t} { \left( \Delta r
\right) ^2 } + m ^2 \frac {(\Delta t) ^2} {4}. \nonumber
\end{eqnarray}

To verify this claim, notice first that (\ref{EasyScheme1}) can be
rewritten as
\begin{equation}
\begin{array}{rcl}
\displaystyle {\frac {v _j ^{n + 1} - 2 v _j ^n + v _j ^{n - 1}}
{(\Delta t) ^2} - \frac {\delta _0 ^2 v _j ^n} {(\Delta r) ^2} +
\gamma \frac {v _j ^{n + 1} - v _j ^{n - 1}} {2 \Delta t}} \ -
\qquad & & \\
\displaystyle {\beta \frac {\delta _0 ^2 v _j ^{n + 1} - \delta _0 ^2 v
_j ^{n - 1}} {2 \Delta t \left( \Delta r \right) ^2 } + \frac {m
^2} {2} \left[ v _j ^{n + 1} + v _j ^{n - 1} \right]} & = & 0.
\end{array} \nonumber
\end{equation}

Define $R = \Delta t / \Delta r$. Let $V _{1j} ^{n + 1} = v _j ^{n
+ 1}$ and $V _{2j} ^{n + 1} = v _j ^n$ for each $j = 0 , 1 , \dots
, M$ and $n = 0, 1, \dots, N - 1$. For every $j = 0, 1, \dots , M$
and $n = 1, 2, \dots , N$ let $\bar {V} _j ^n$ be the column
vector whose components are $V _{1j} ^n$ and $V _{2j} ^n$. Our
problem can be written then in matricial form as
\begin{eqnarray}
\left( \begin{array}{cc} %
k & 0 \\
0 & 1 \end{array} \right)
\bar {V} _j ^{n + 1} & = & \left( %
\begin{array}{cc}
2 + R ^2 \delta _0 ^2 & - h \\
1 & 0
\end{array}
\right) \bar {V} _j ^n, \nonumber
\end{eqnarray}
where
\begin{eqnarray}
k & = & 1 + \gamma \frac {\Delta t} {2} - \frac {\beta \Delta t
\delta _0 ^2} {2 \left( \Delta r \right) ^2} + m ^2 \frac {( \Delta
t) ^2} {2} \qquad \qquad {\rm and} \nonumber \\
h & = & 1 - \gamma \frac {\Delta t} {2} + \frac {\beta \Delta t
\delta _0 ^2} {2 \left( \Delta r \right) ^2} + m ^2 \frac {( \Delta
t) ^2} {2}. \nonumber
\end{eqnarray}

Denoting the Fourier transform of each $\bar {V} _j ^{n}$ by $\hat
V _j ^n$, we obtain that
\begin{eqnarray}
\hat {V} _j ^{n + 1} & = & \left( %
\begin{array}{cc}
\frac {2} {\hat {k} ( \xi )} \left(1 - 2 R ^2 \sin ^2 \frac {\xi}
{2} \right) & - \frac {\hat {h} ( \xi )} {\hat {k} ( \xi )} \\
1 & 0
\end{array}
\right) \hat {V} _j ^n, \nonumber
\end{eqnarray}
where
\begin{eqnarray}
\hat {k} ( \xi ) & = & 1 + \gamma \frac {\Delta t} {2} + 2 \frac
{\beta \Delta t} {\left( \Delta r \right) ^2} \sin ^2 \frac { \xi
} {2} + m ^2 \frac {( \Delta
t) ^2} {2} \qquad \qquad {\rm and} \nonumber \\
\hat {h} ( \xi ) & = & 1 - \gamma \frac {\Delta t} {2} - 2 \frac
{\beta \Delta t} {\left( \Delta r \right) ^2} \sin ^2 \frac { \xi
} {2} + m ^2 \frac {( \Delta t) ^2} {2}. \nonumber
\end{eqnarray}

The matrix $A ( \xi )$ multiplying $\hat {V} _j ^n$ in the last
vector equation is the amplification matrix of the problem, which
has eigenvalues given by
\begin{eqnarray}
\lambda _\pm & = & \frac {1 - 2 R ^2 \sin ^2 \frac {\xi} {2} \pm
\sqrt{ \left( 1 - 2 R ^2 \sin ^2 \frac {\xi} {2} \right) - \hat
{h} ( \xi ) \hat {k} ( \xi ) }} {\hat {k} ( \xi )}. \nonumber
\end{eqnarray}
In particular, for $\xi = \pi$ the eigenvalues of $A$ are
\begin{eqnarray}
\lambda _\pm & = & \frac {1 - 2 R ^2 \pm \sqrt{ (1 - 2 R ^2) ^2 -
\hat {h} (\pi) \hat {k} (\pi)}} {\hat {k} (\pi)}. \nonumber
\end{eqnarray}

Suppose for a moment that $1 - 2 R ^2 < - \hat {k} (\pi)$. If the
radical in the expression for the eigenvalues of $A (\pi)$ is a
pure real number then $| \lambda _- | > 1$. So for every $n \in
\mathbb {N}$, $|| A ^n || \geq | \lambda _- | ^n$ grows faster
than $K _1 + n K _2$ for any constants $K _1$ and $K _2$. A
similar situation happens when the radical is a pure imaginary
number, except that in this case $| \cdot |$ represents the usual
Euclidean norm in the field of complex numbers.

\smallskip
Summarizing, if $1 - 2 R ^2 < - \hat {k} (\pi)$ then scheme
(\ref{EasyScheme1}) is unstable. Therefore in order for
our numeric method to be stable order $n$ it is necessary that $1
- 2 R ^2 > - \hat {k} (\pi)$, which is what we needed to
establish.

\subsection*{Energy analysis}

\begin{figure}[tcb]
\centerline{
\begin{tabular}{cc}
\scriptsize{$t = 0$} &
\scriptsize{$t = 0.04$} \\
\includegraphics[width=0.45\textwidth]{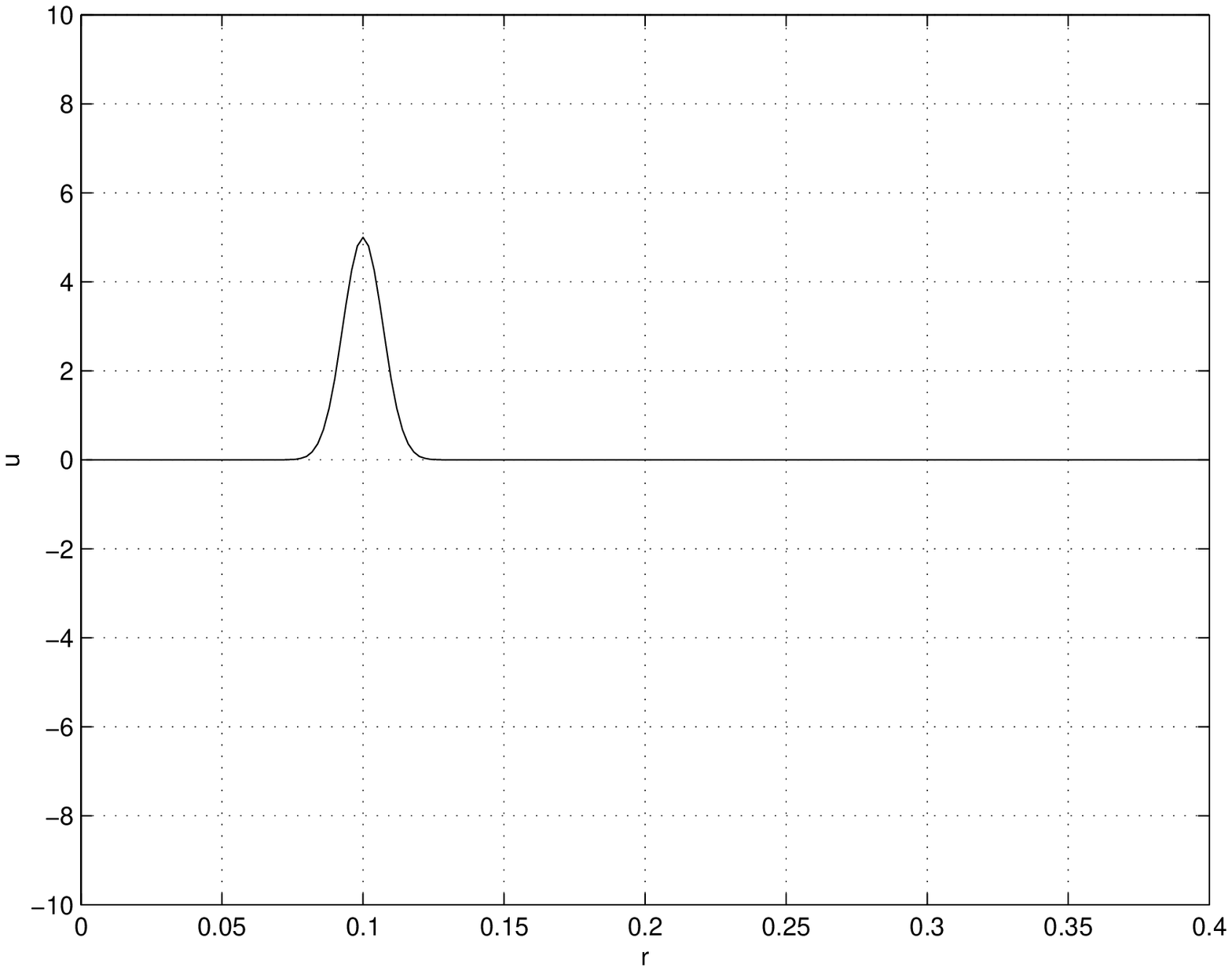} &
\includegraphics[width=0.45\textwidth]{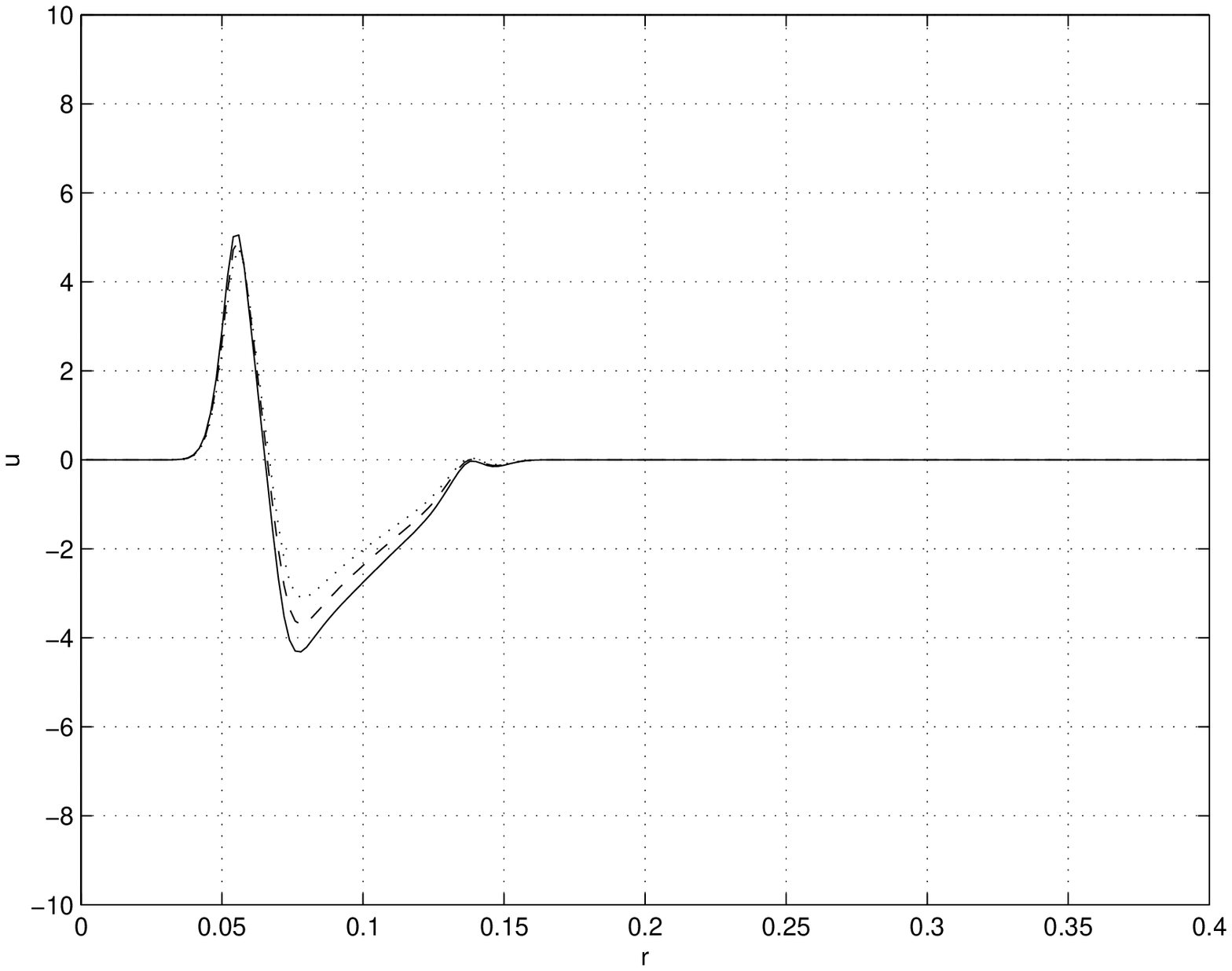} \\
\scriptsize{$t = 0.08$} &
\scriptsize{$t = 0.12$} \\
\includegraphics[width=0.45\textwidth]{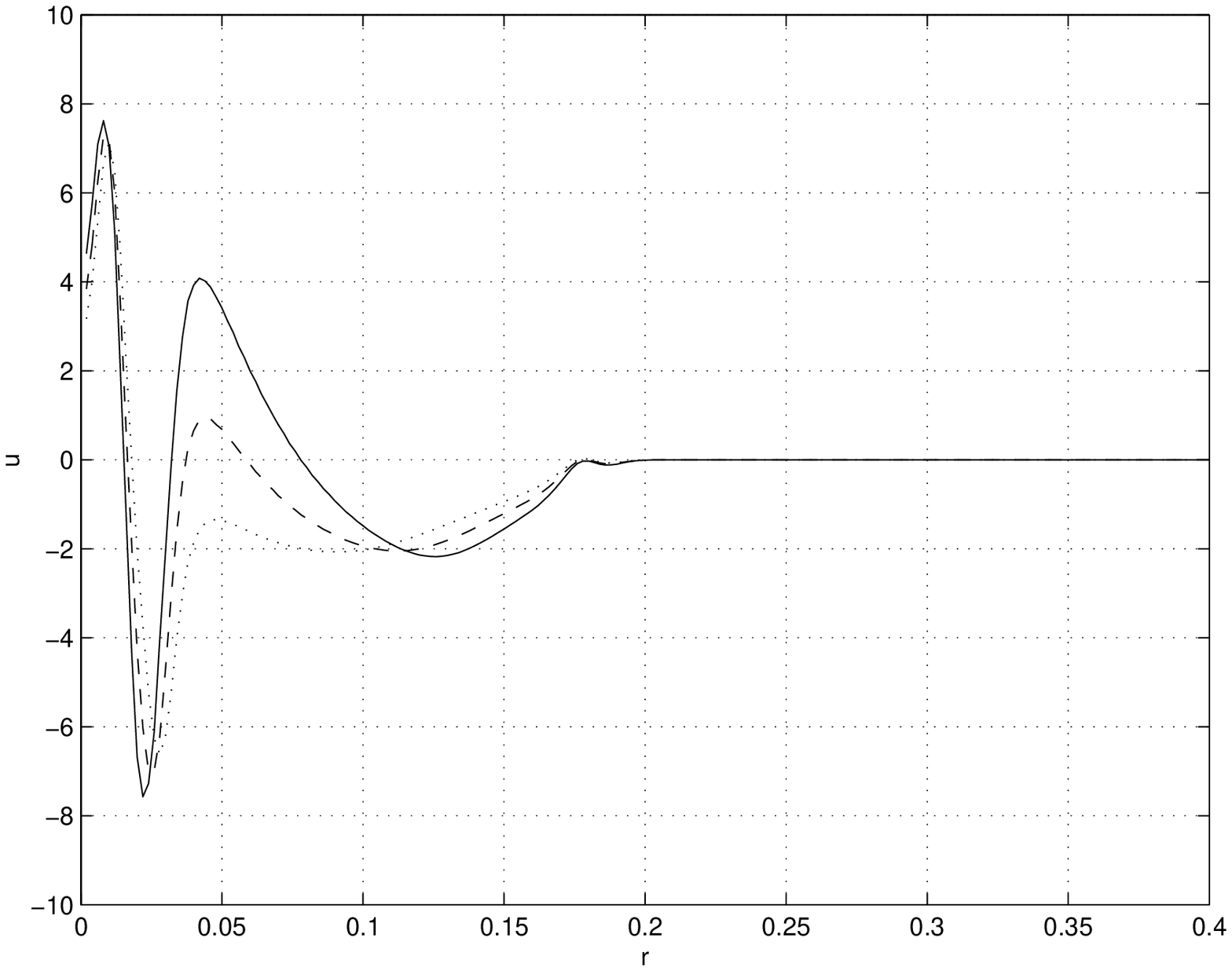} &
\includegraphics[width=0.45\textwidth]{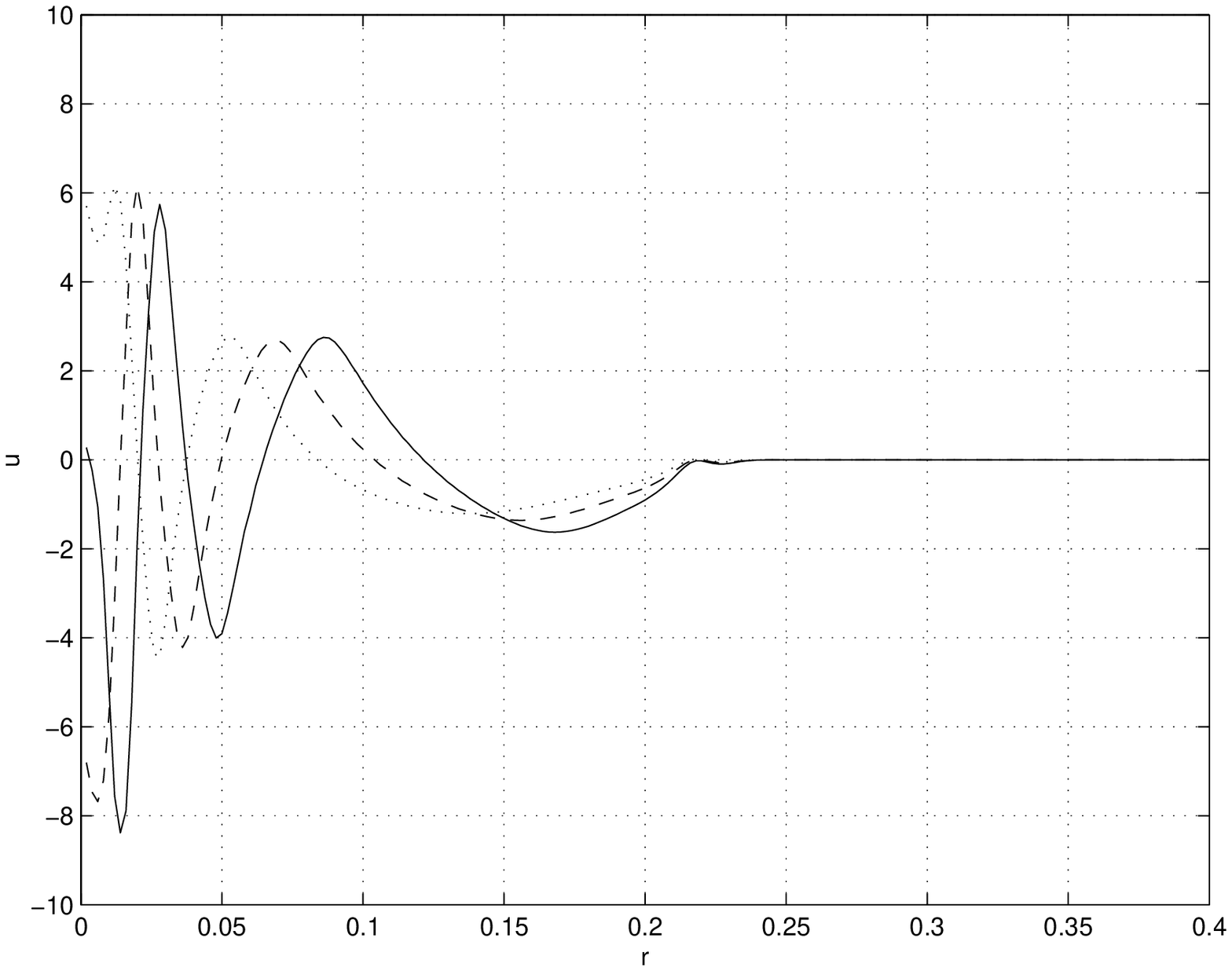} \\
\scriptsize{$t = 0.16$} &
\scriptsize{$t = 0.2$} \\
\includegraphics[width=0.45\textwidth]{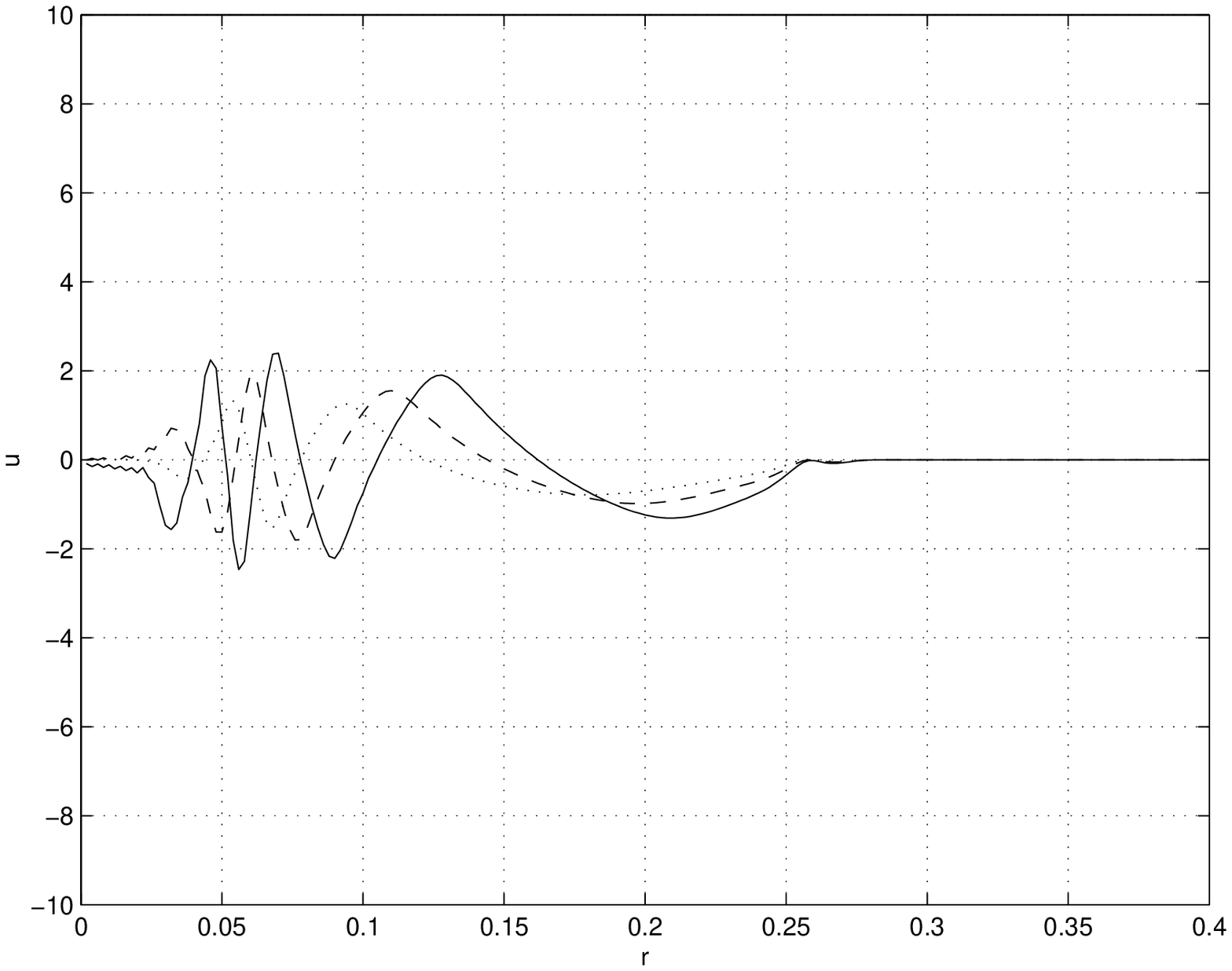} &
\includegraphics[width=0.45\textwidth]{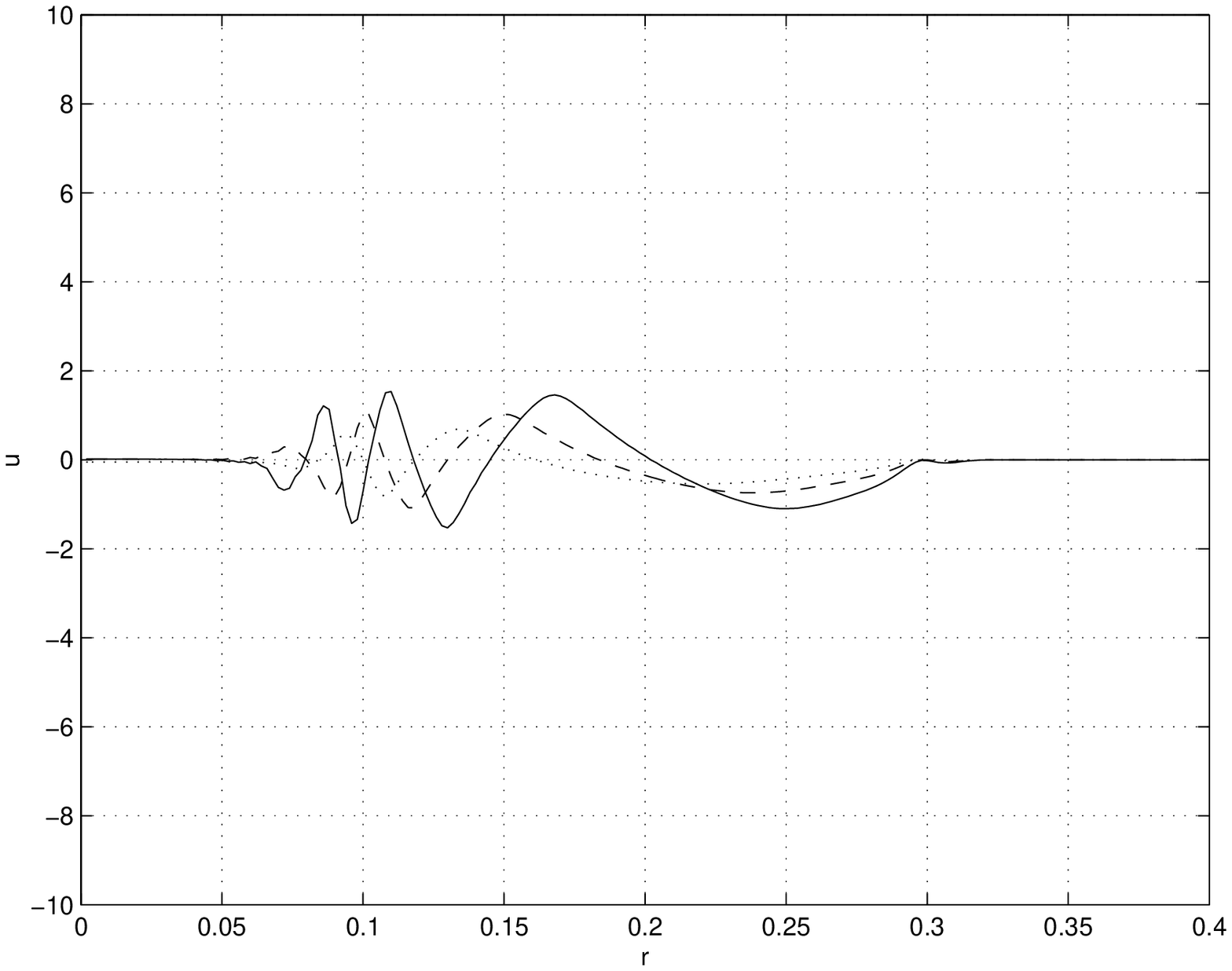} \\
\end{tabular}}
\caption{Approximate radial solutions of (\ref{paperproblem}) at
successive times for $\beta = 0$ and values of $\gamma = 0$ (solid), $\gamma = 5$ (dashed)
and $\gamma = 10$ (dotted), $G ^\prime (u) = u ^7$, and initial
data $\phi (r) = h (r)$, $\psi (r) = h ^\prime (r) + h (r) / r$.
\label{Fig4-2}} %
\end{figure}

\begin{figure}[tcb]
\centerline{
\begin{tabular}{cc}
\scriptsize{$G ^\prime (u) = 0$} &
\scriptsize{$G ^\prime (u) = u ^3$} \\
\includegraphics[width=0.45\textwidth]{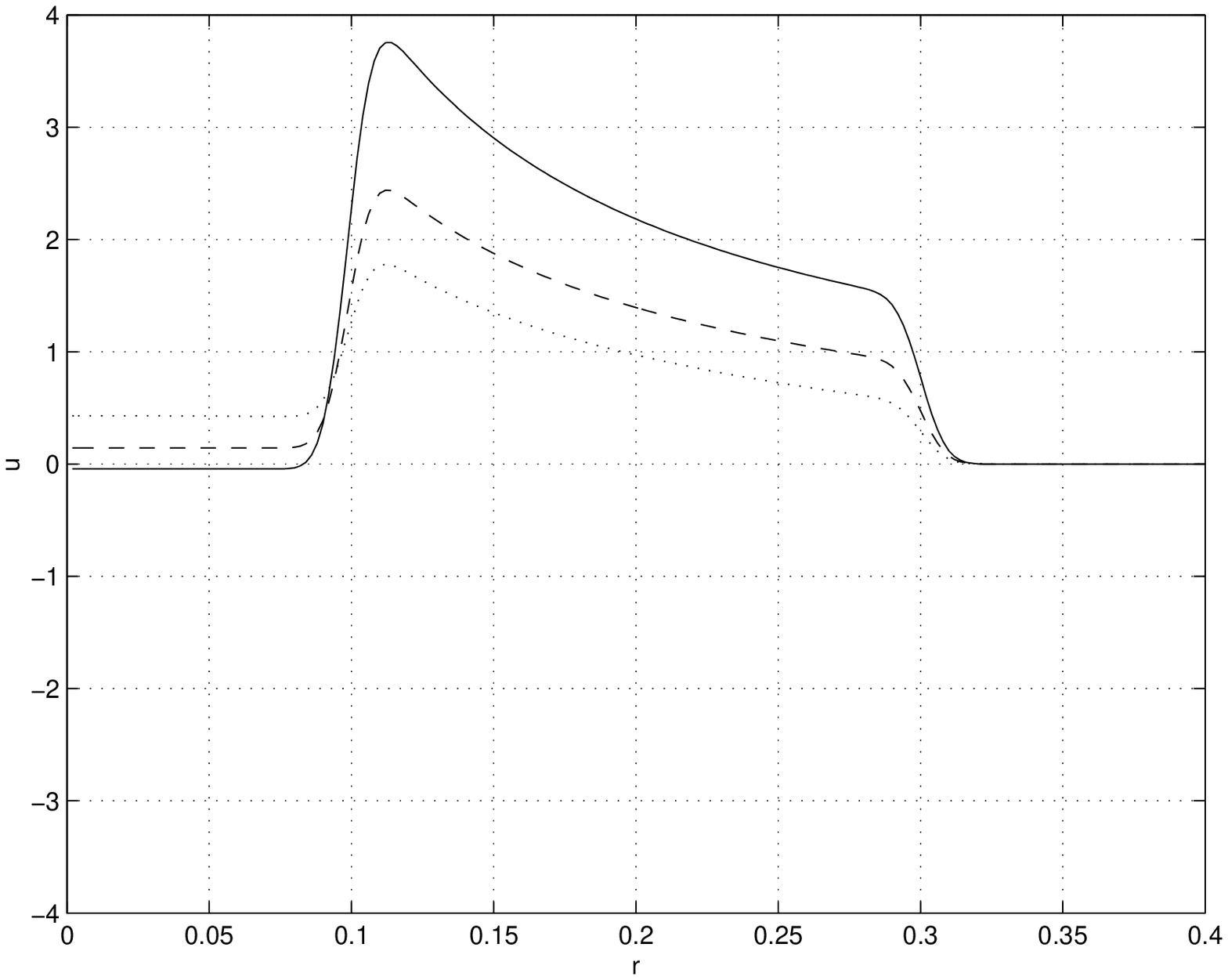} &
\includegraphics[width=0.45\textwidth]{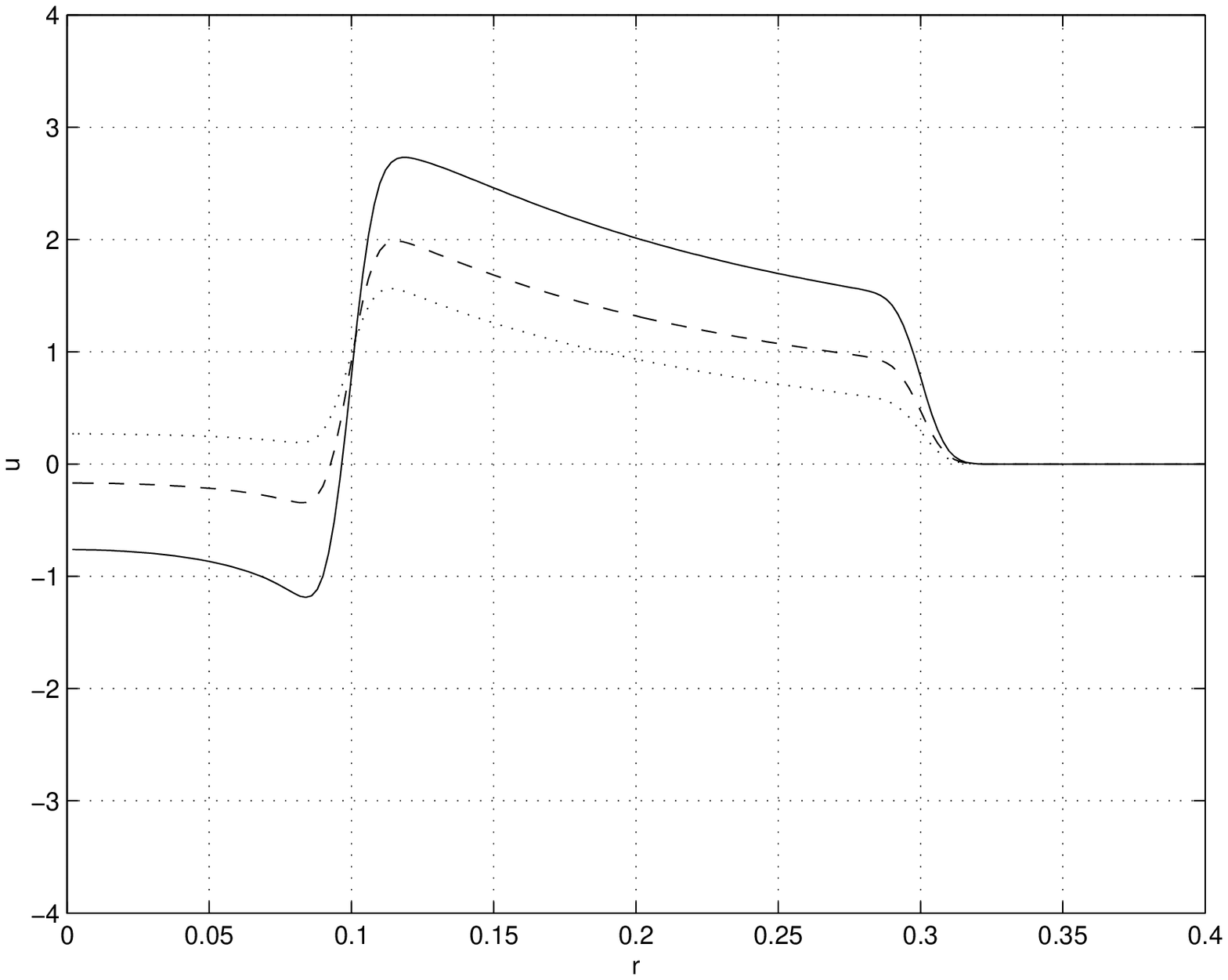} \\
\scriptsize{$G ^\prime (u) = u ^5$} &
\scriptsize{$G ^\prime (u) = u ^7$} \\
\includegraphics[width=0.45\textwidth]{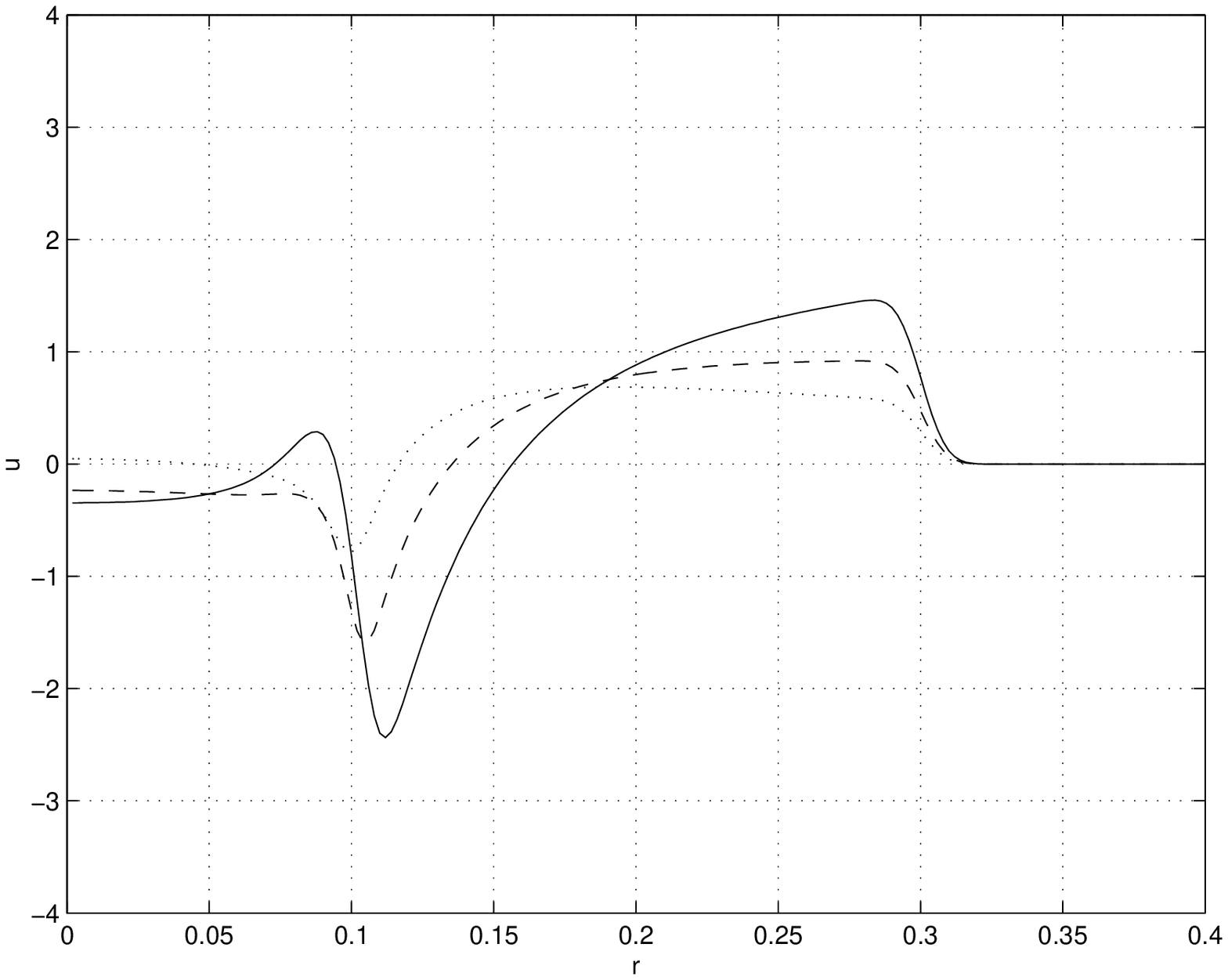} &
\includegraphics[width=0.45\textwidth]{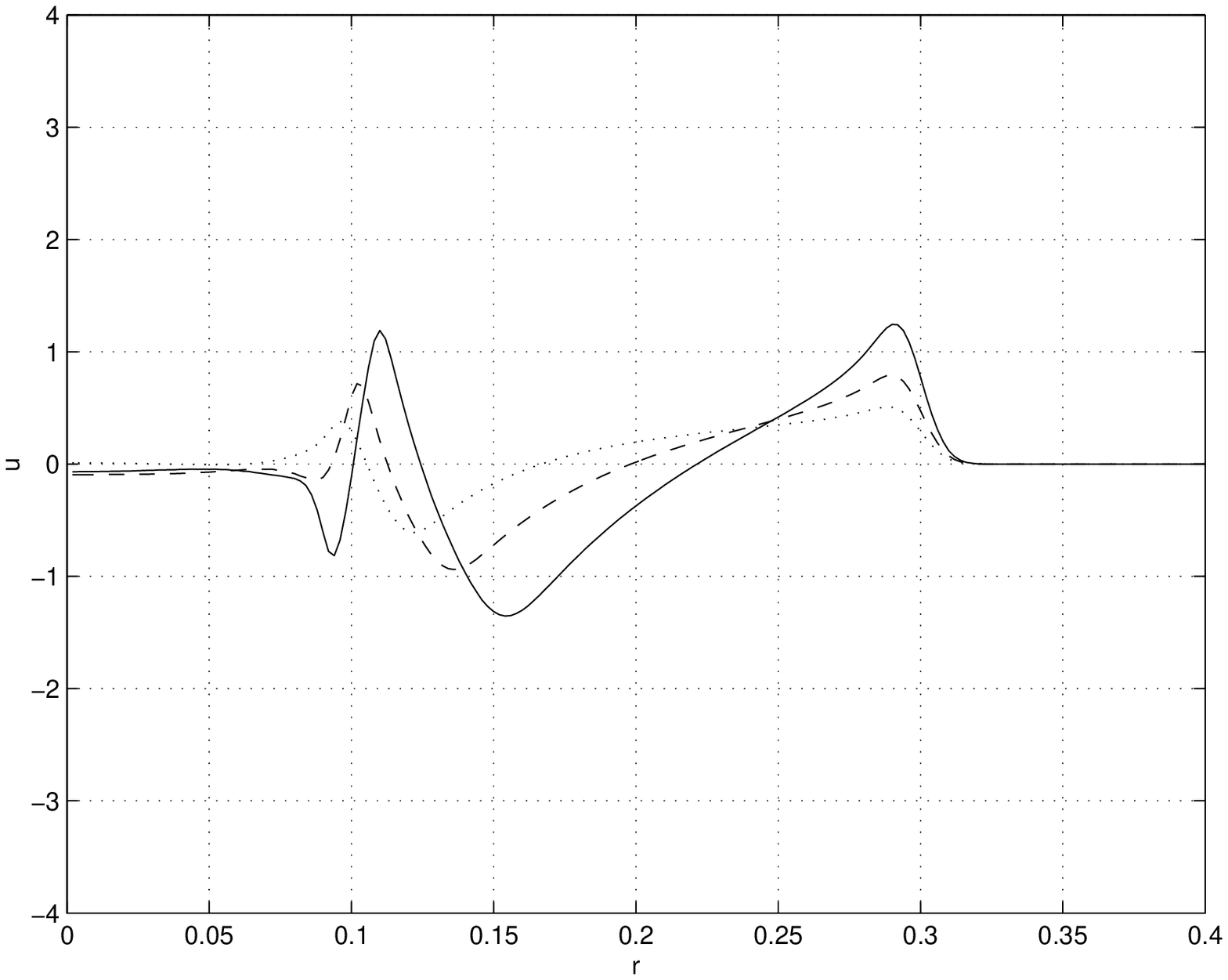} \\
\scriptsize{$G ^\prime (u) = u ^9$} &
\scriptsize{$G ^\prime (u) = \sin (5u) - 5u$} \\
\includegraphics[width=0.45\textwidth]{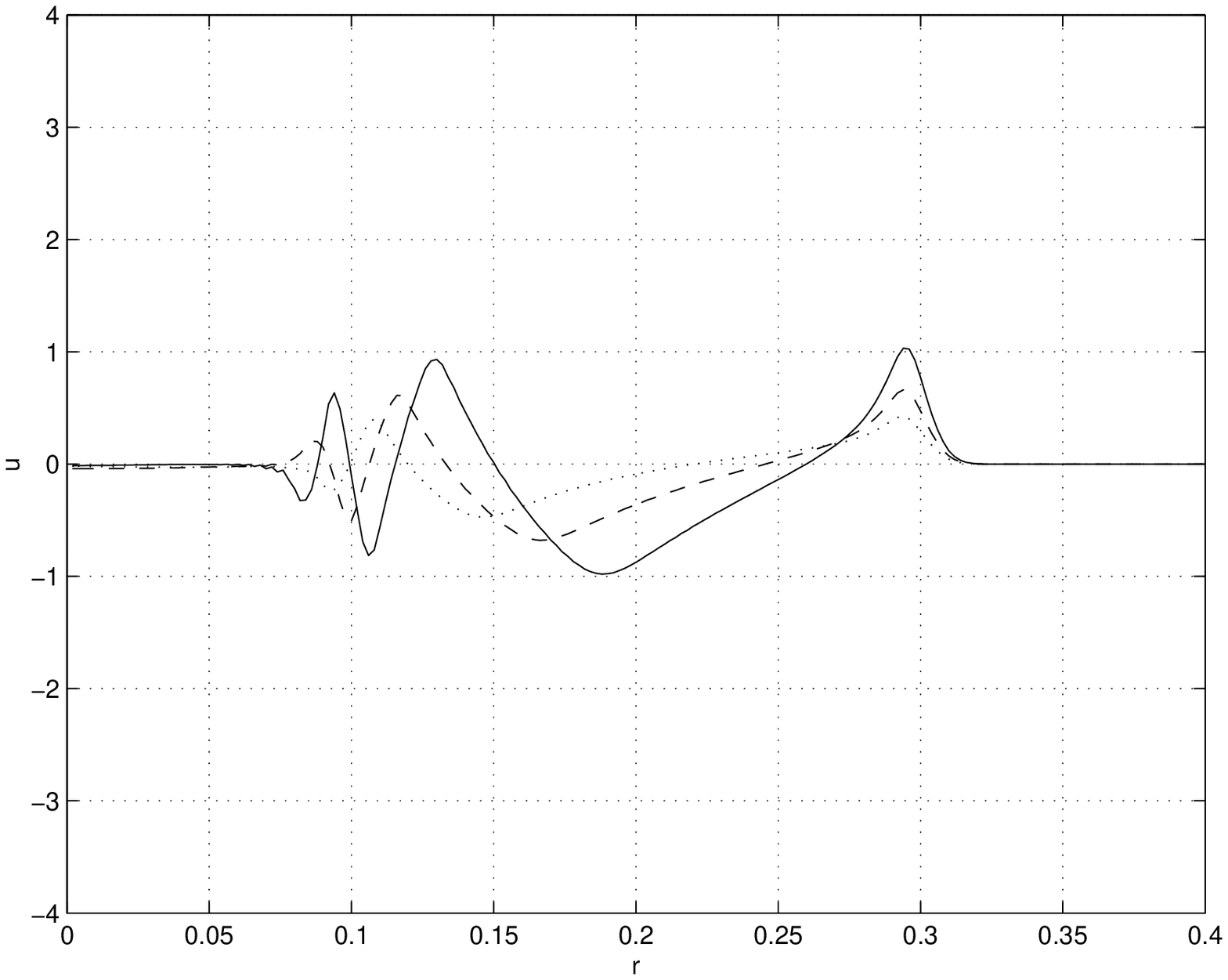} &
\includegraphics[width=0.45\textwidth]{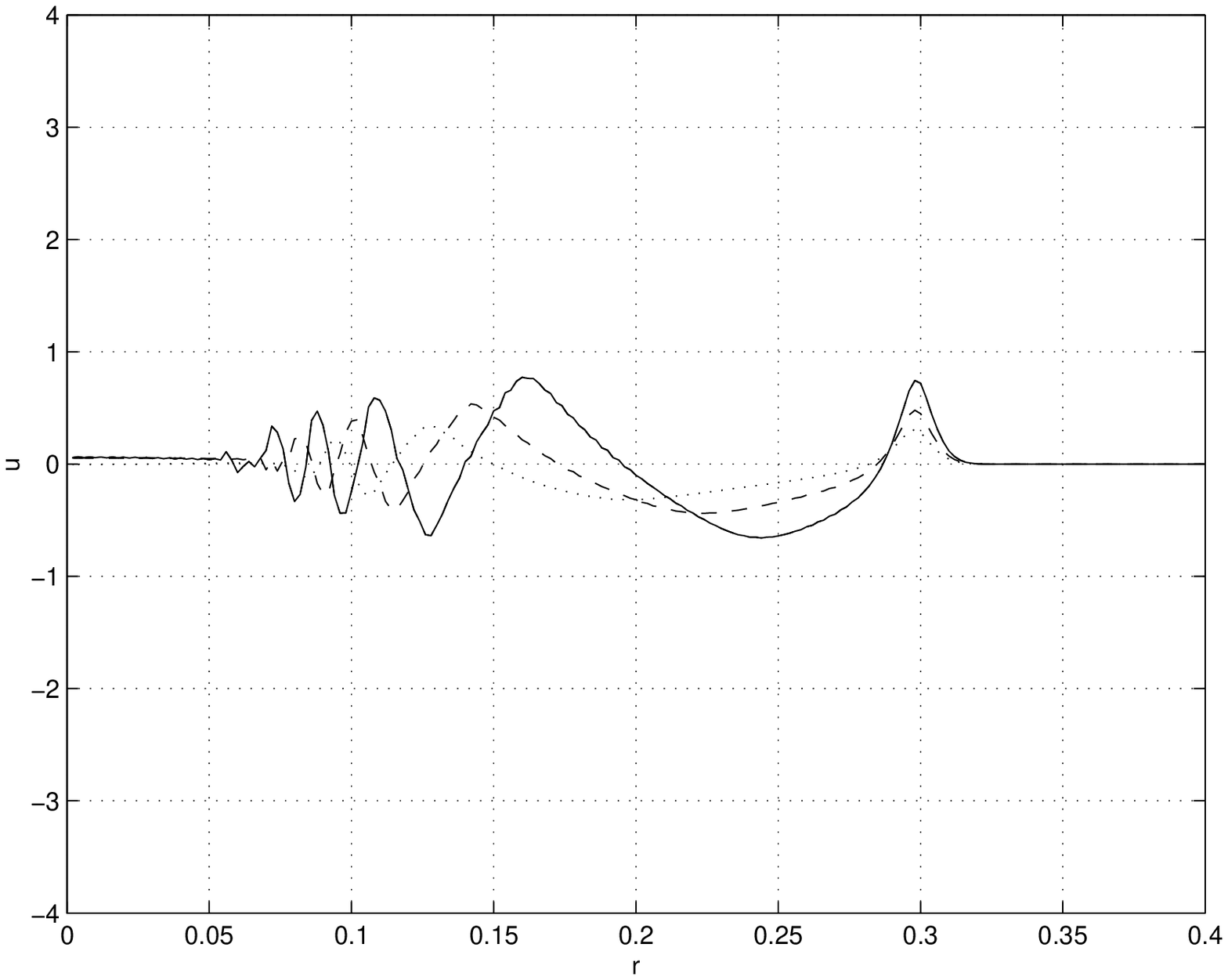} \\
\end{tabular}}
\caption{Approximate radial solutions of (\ref{paperproblem}) with
$G ^\prime (u)$ at $t = 0.2$, for initial data $\phi (r) = 0$ and
$\psi (r) = 100 h (r)$, $\beta = 0$ and $\gamma = 0$ (solid), $\gamma = 5$
(dashed) and $\gamma = 10$ (dotted).
\label{Fig4-3}} %
\end{figure}

\begin{figure}[tcb]
\centerline{
\begin{tabular}{cc}
\scriptsize{$t = 0$} &
\scriptsize{$t = 0.04$} \\
\includegraphics[width=0.5\textwidth]{Grafica11.eps} &
\includegraphics[width=0.5\textwidth]{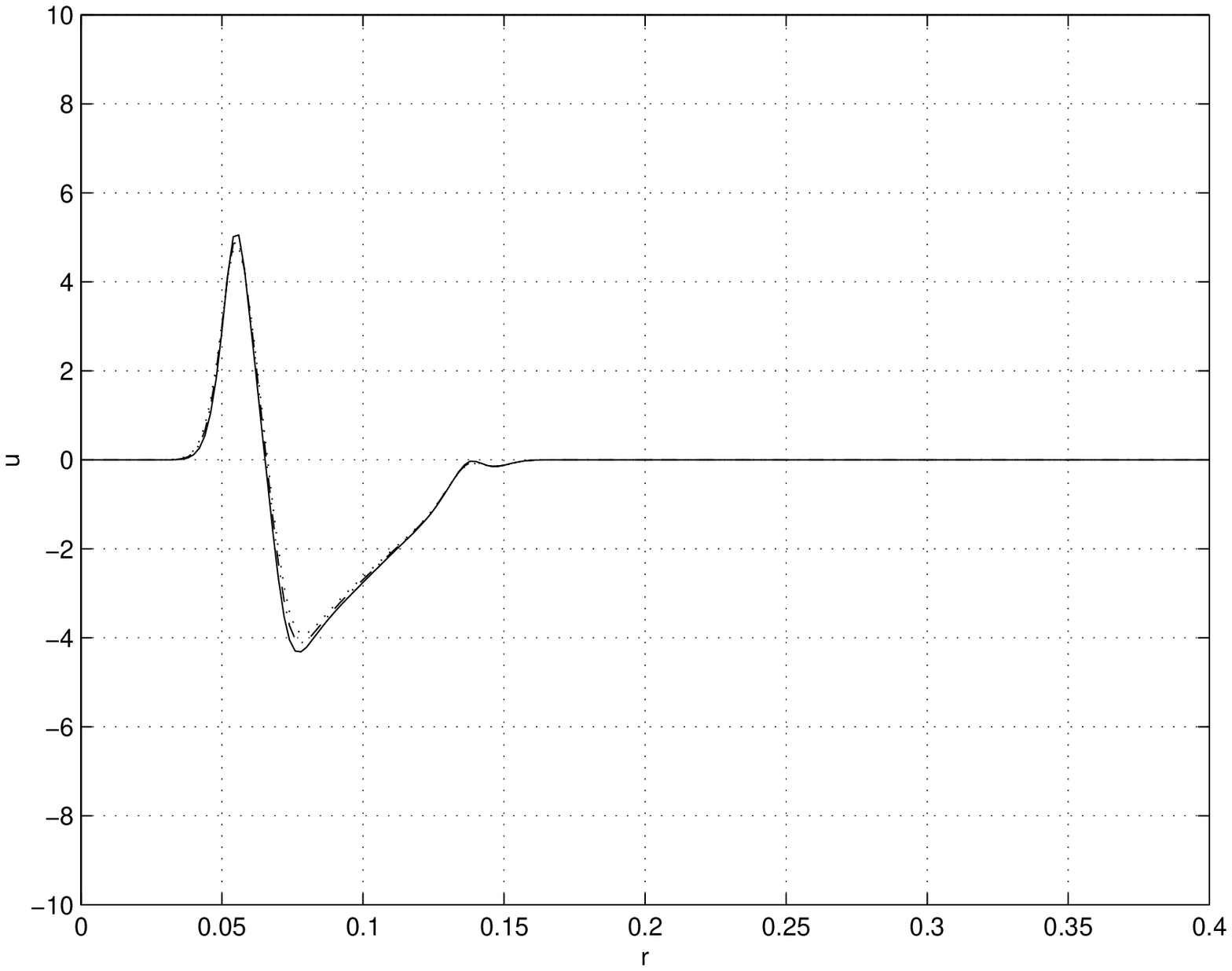} \\
\scriptsize{$t = 0.08$} &
\scriptsize{$t = 0.12$} \\
\includegraphics[width=0.5\textwidth]{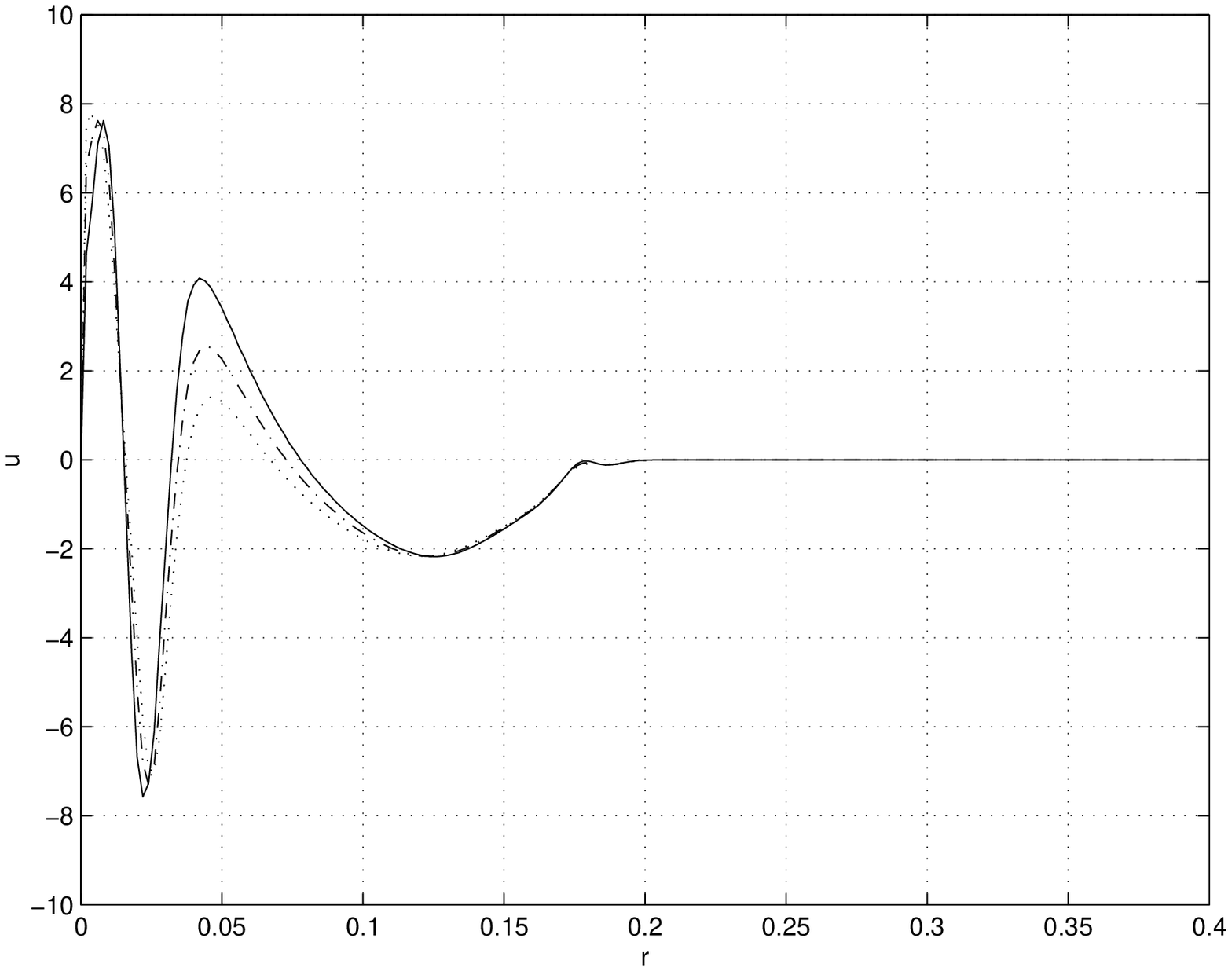} &
\includegraphics[width=0.5\textwidth]{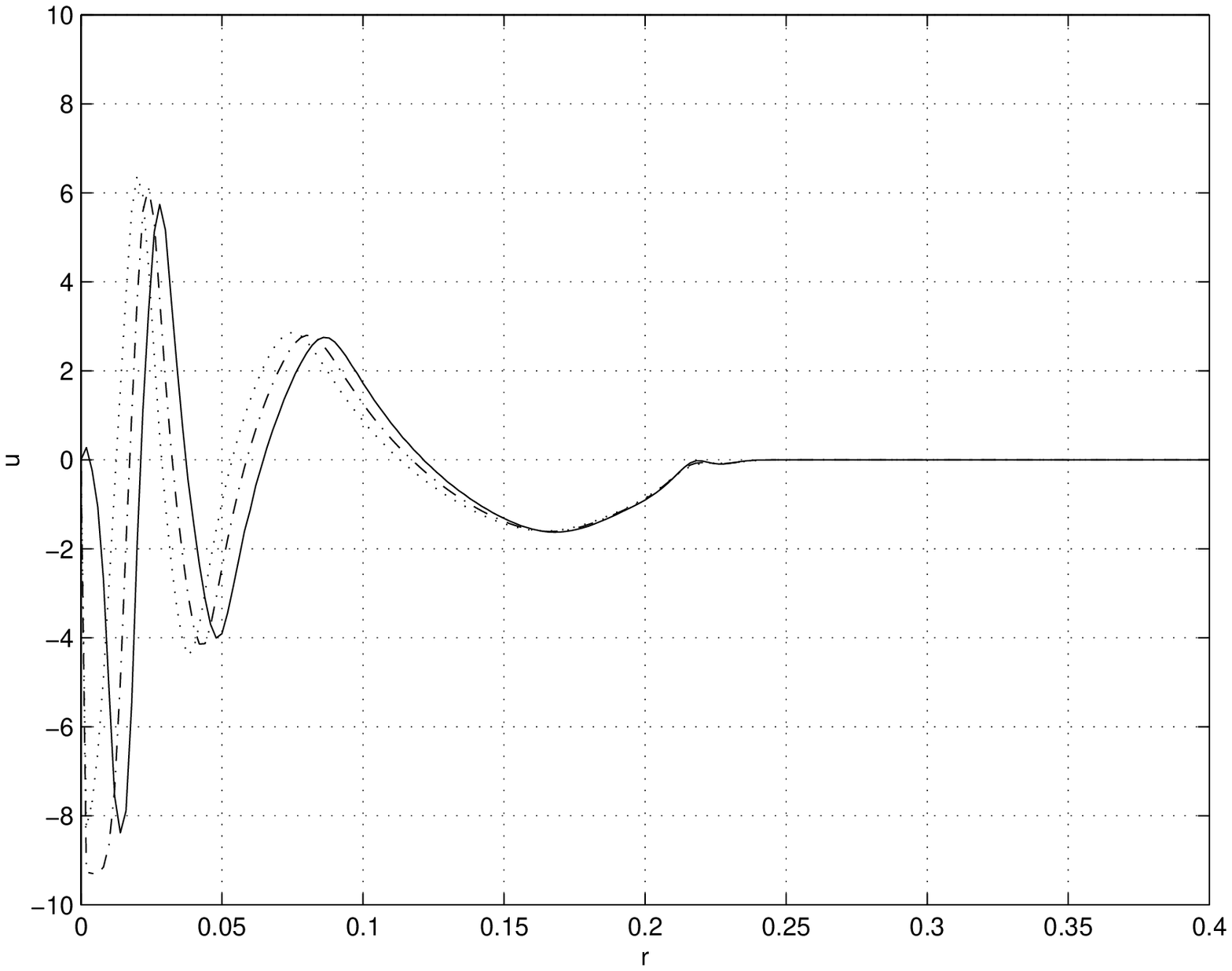} \\
\scriptsize{$t = 0.16$} &
\scriptsize{$t = 0.2$} \\
\includegraphics[width=0.5\textwidth]{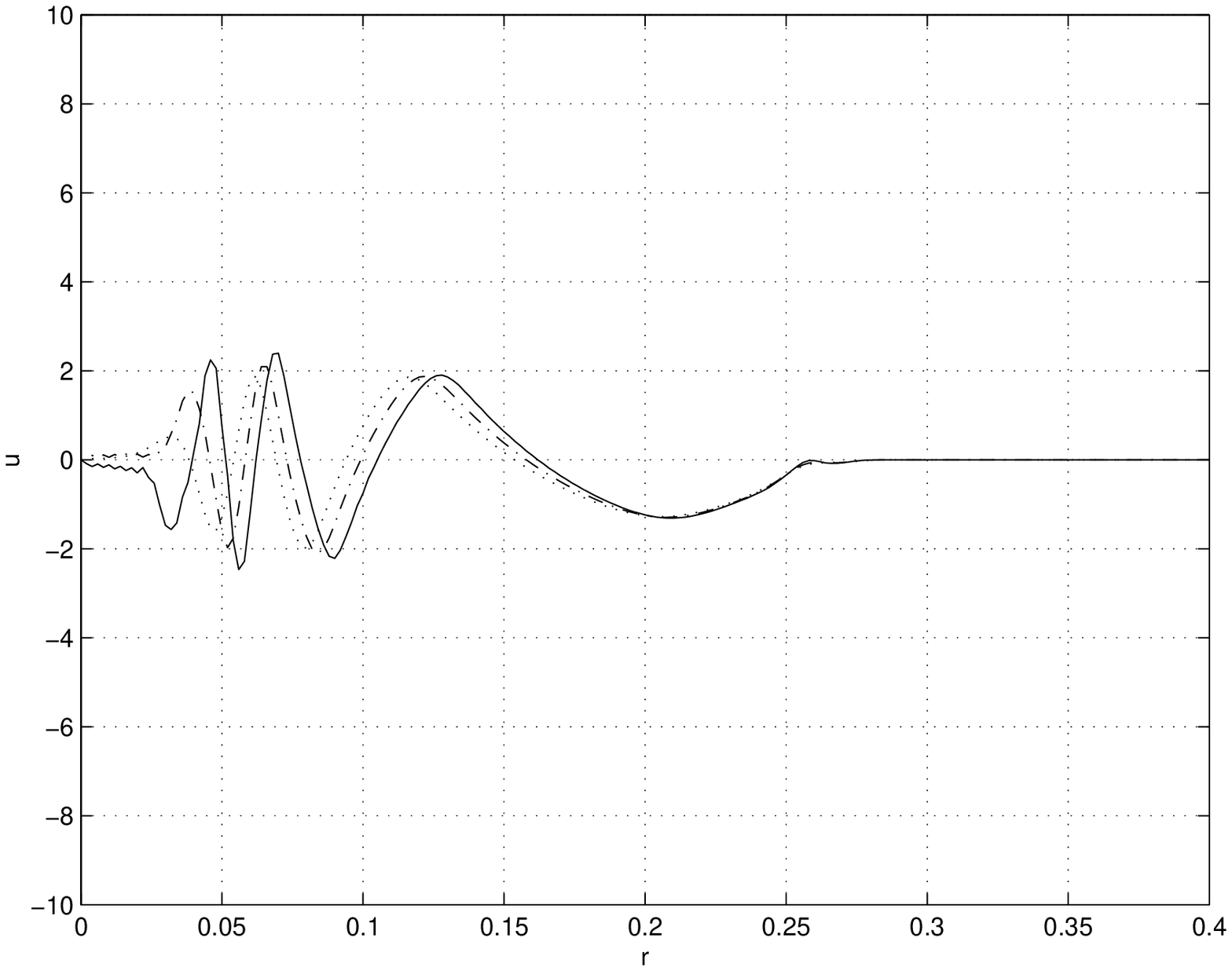} &
\includegraphics[width=0.5\textwidth]{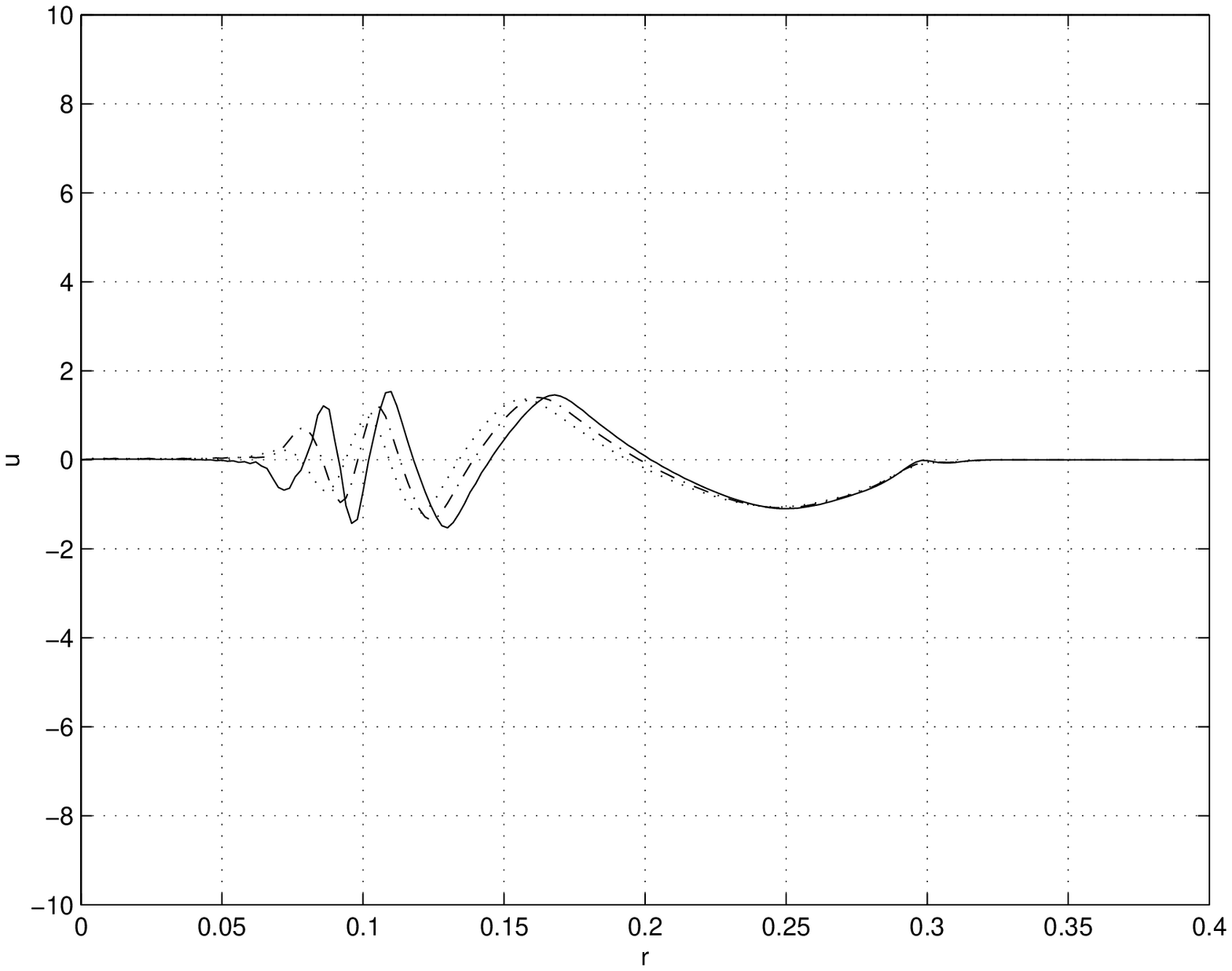} \\
\end{tabular}}
\caption{Approximate radial solutions at successive times of the
undamped (solid) and the damped nonlinear Klein-Gordon equation
with $\gamma=0$ and $\beta = 0$ (dashed), $\beta = 0.0001$ (dash-dotted) and
$\beta = 0.0002$, nonlinear term $G ^\prime (u) = u ^7$, and initial
data $\phi (r) = h (r)$ and $\psi (r) = h ^\prime (r) + h (r) /
r$.\label{Fig4-2-2}}
\end{figure}

\begin{figure}[tcb]
\centerline{
\begin{tabular}{cc}
\scriptsize{$G ^\prime (u) = u ^3$} &
\scriptsize{$G ^\prime (u) = u ^3$} \\
\includegraphics[width=0.45\textwidth]{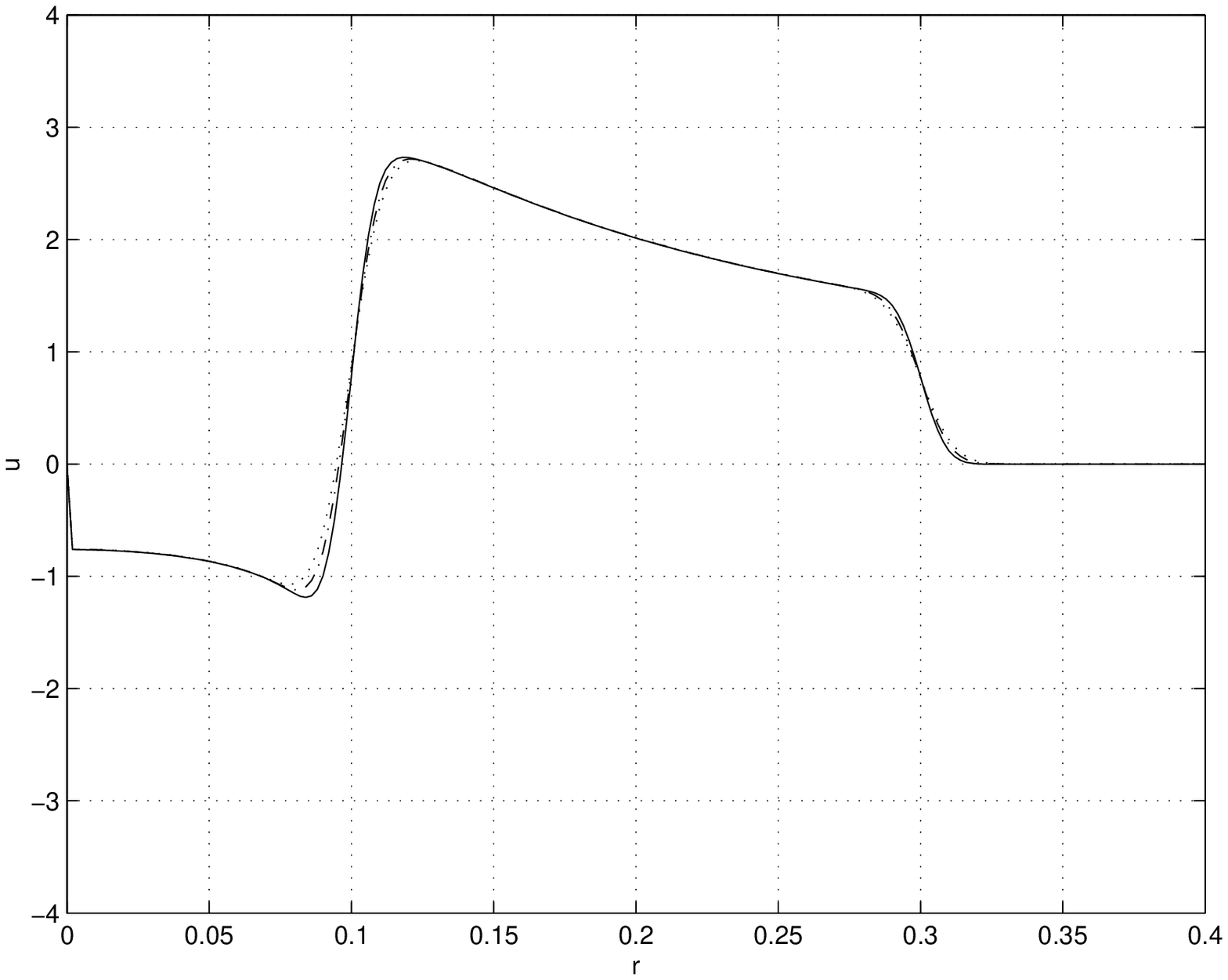} &
\includegraphics[width=0.45\textwidth]{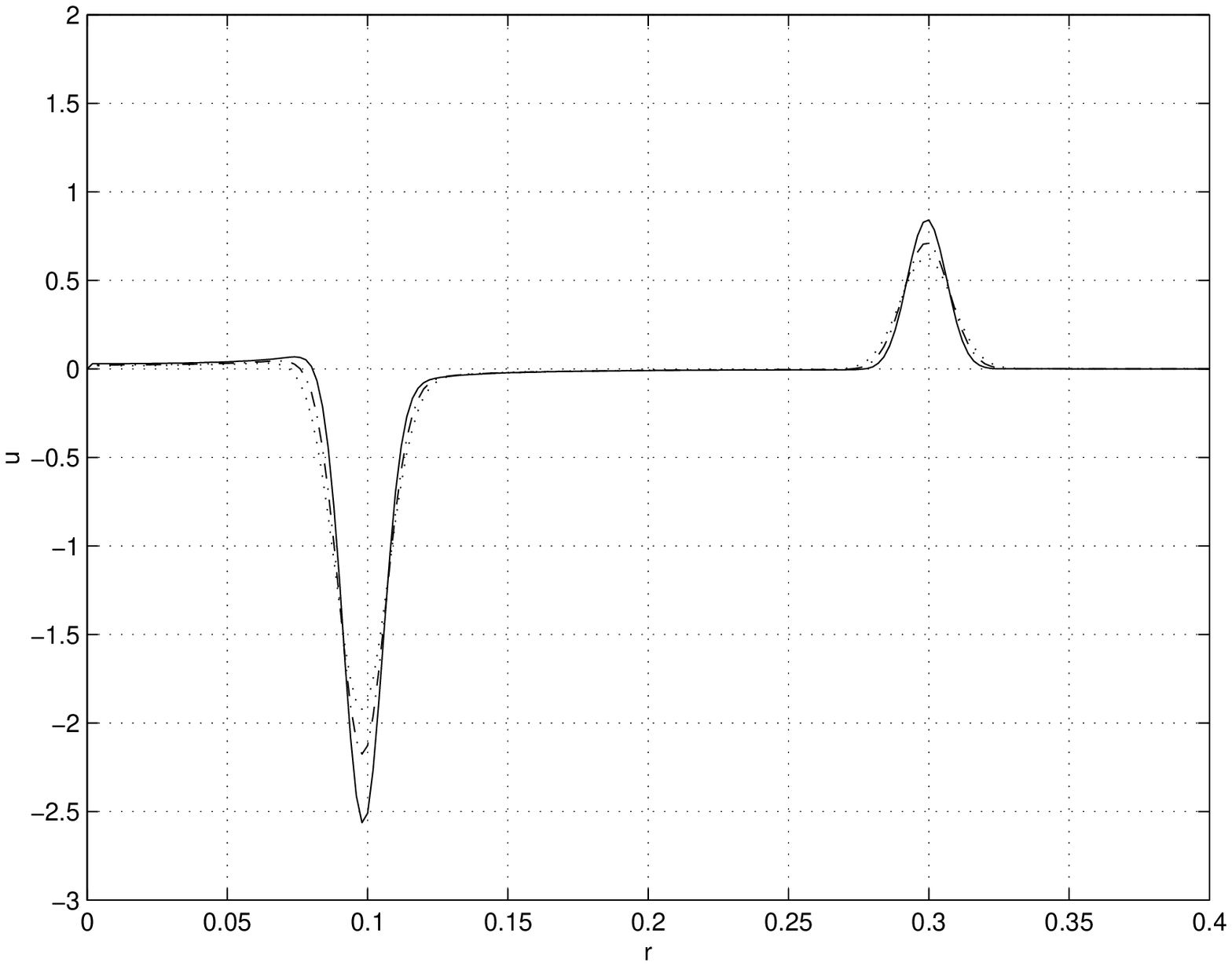} \\
\scriptsize{$G ^\prime (u) = u ^5$} &
\scriptsize{$G ^\prime (u) = u ^5$} \\
\includegraphics[width=0.45\textwidth]{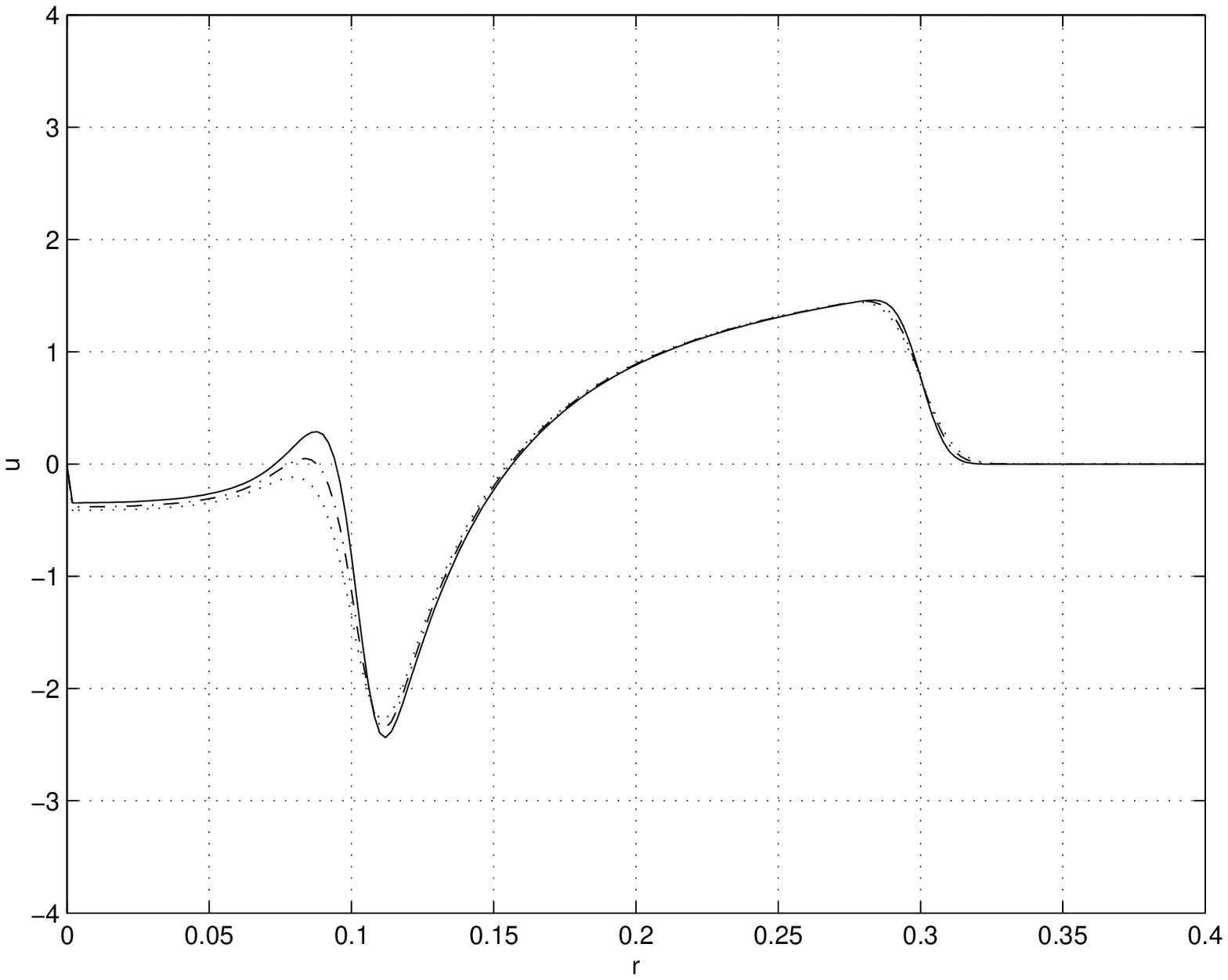} &
\includegraphics[width=0.45\textwidth]{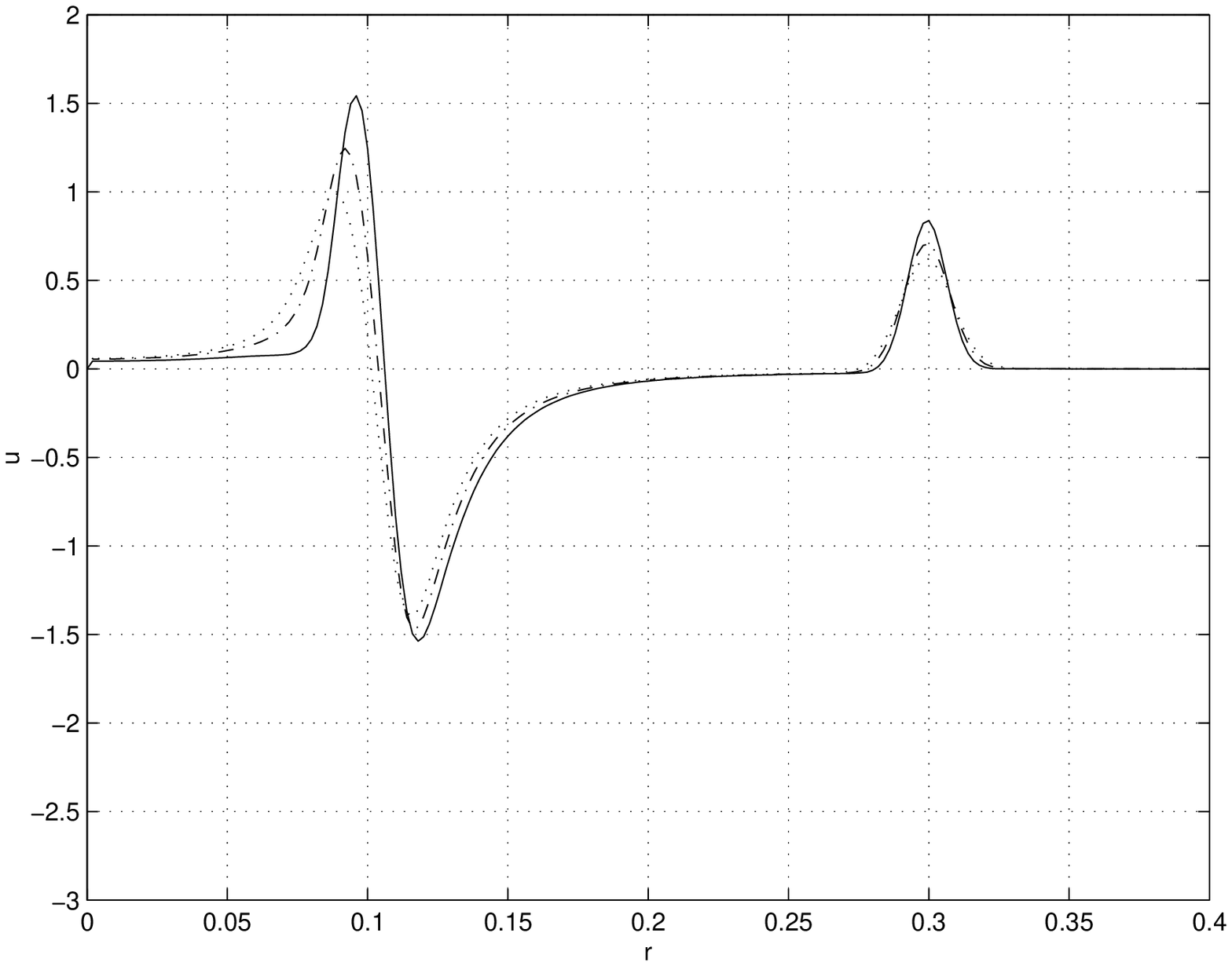} \\
\scriptsize{$G ^\prime (u) = u ^7$} &
\scriptsize{$G ^\prime (u) = u ^7$} \\
\includegraphics[width=0.45\textwidth]{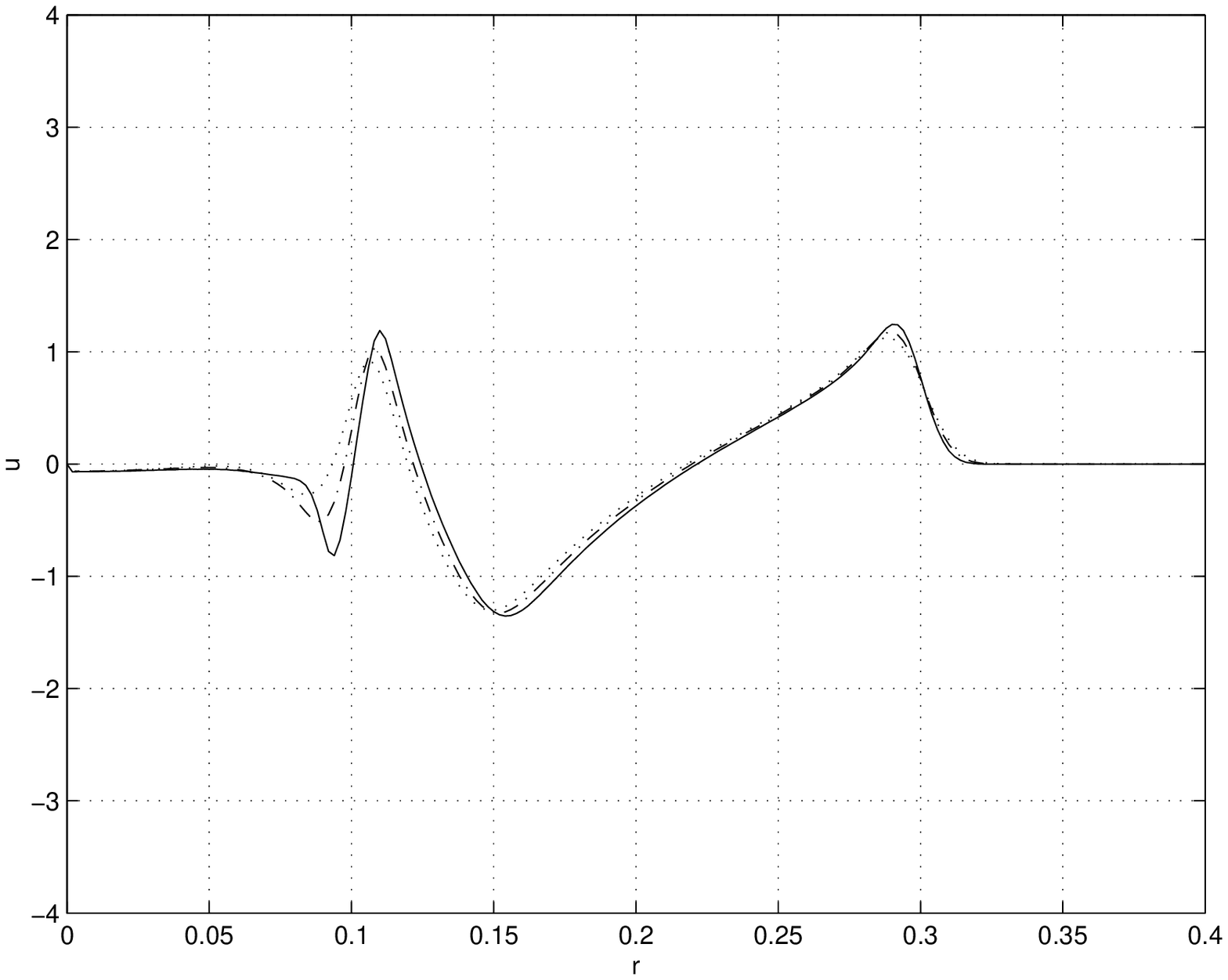} &
\includegraphics[width=0.45\textwidth]{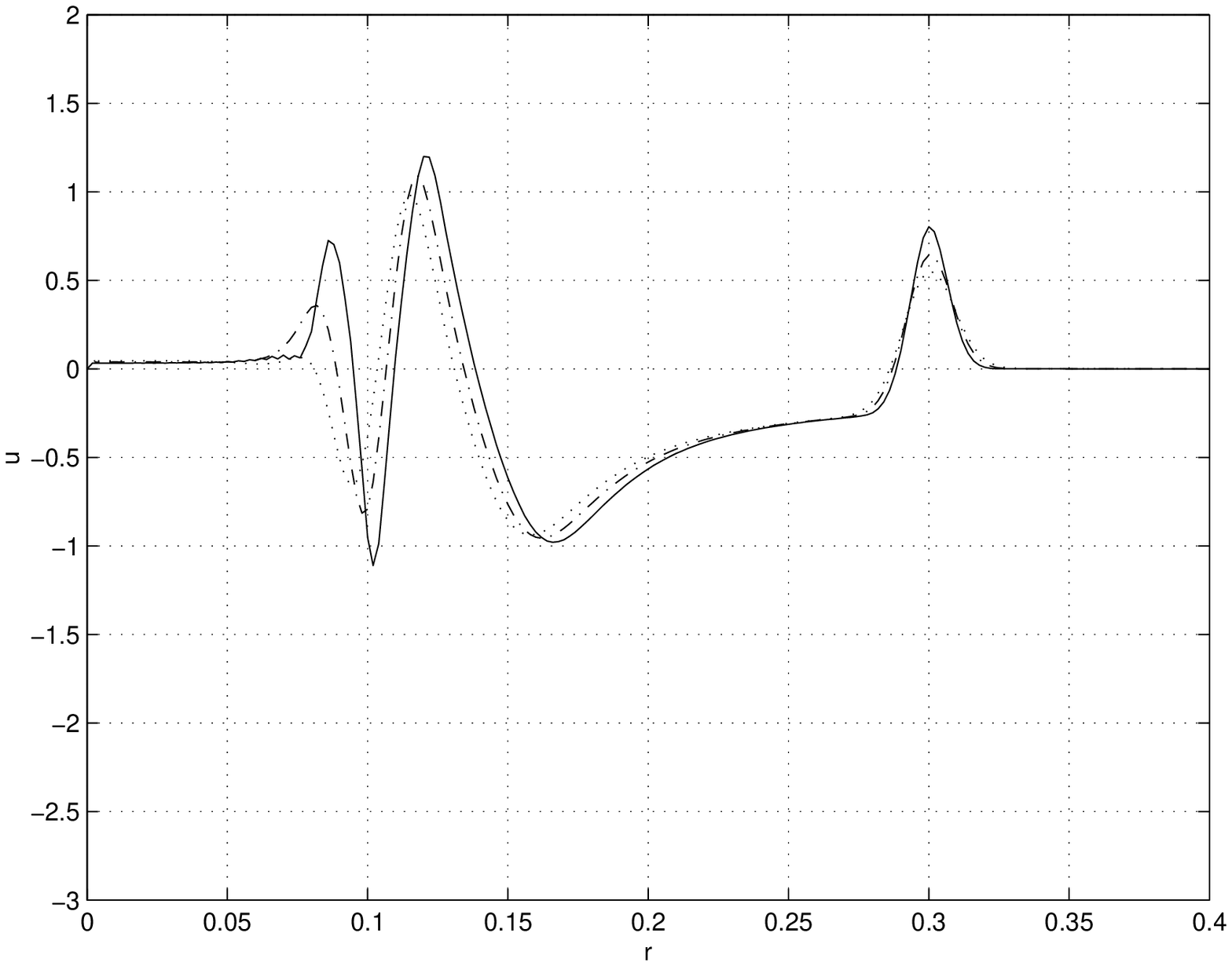} \\
\end{tabular}}
\caption{Approximate radial solutions of (\ref{paperproblem}) with
$G ^\prime (u)$ at $t = 0.2$, for initial data $\phi (r) = 0$,
$\psi (r) = 100 h (r)$ (left column) and $\phi (r) = h (r)$, $\psi
(r) = 0$ (right column), $\gamma = 0$ and $\beta = 0$ (solid),
$\beta = 0.0001$ (dashed) and $\beta = 0.0002$ (dotted). \label{Fig4-2-3-1}} %
\end{figure}

\begin{figure}[tcb]
\centerline{
\begin{tabular}{cc}
\scriptsize{$G ^\prime (u) = 0$} &
\scriptsize{$G ^\prime (u) = u ^3$} \\
\includegraphics[width=0.5\textwidth]{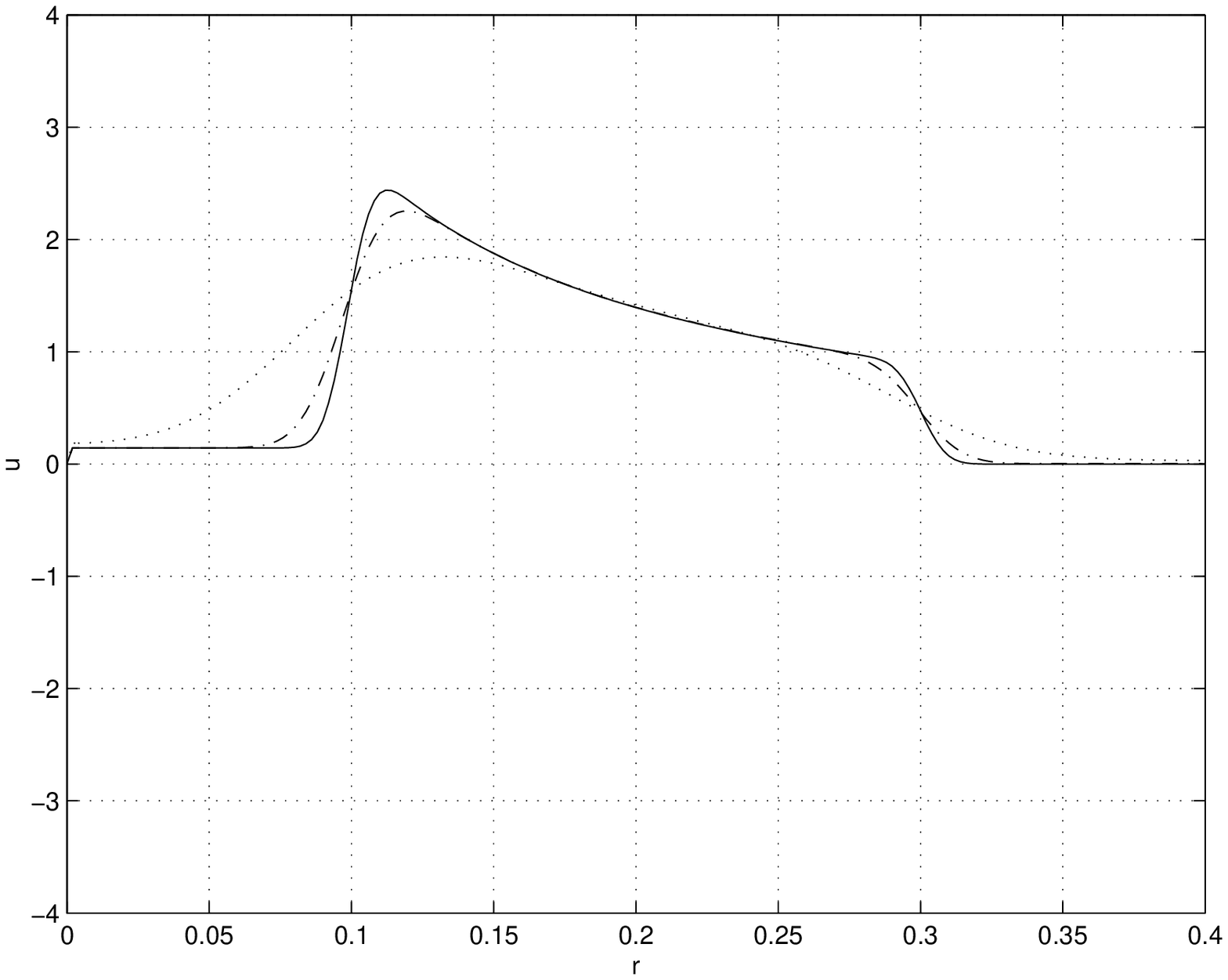} &
\includegraphics[width=0.5\textwidth]{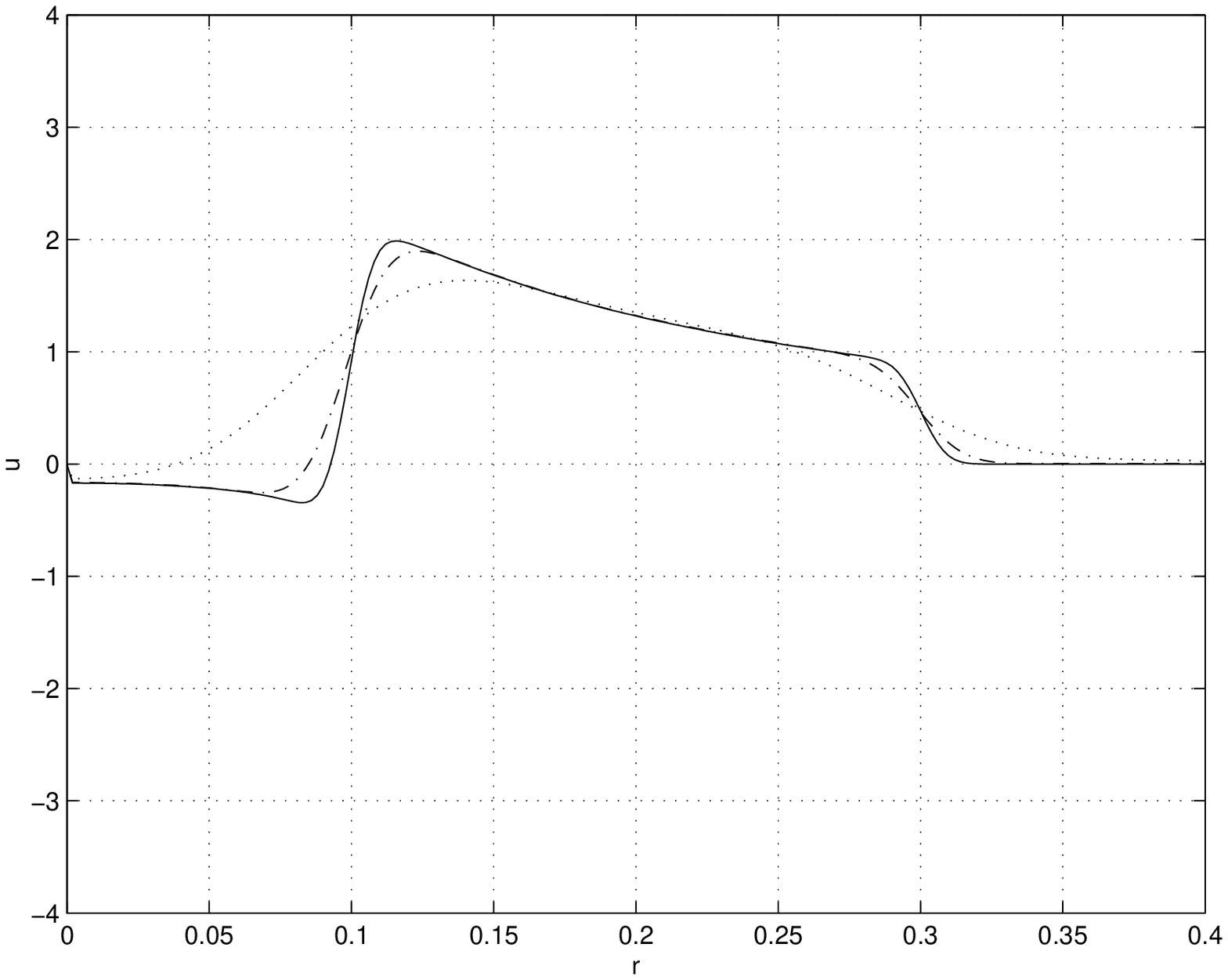} \\
\scriptsize{$G ^\prime (u) = u ^5$} &
\scriptsize{$G ^\prime (u) = u ^7$} \\
\includegraphics[width=0.5\textwidth]{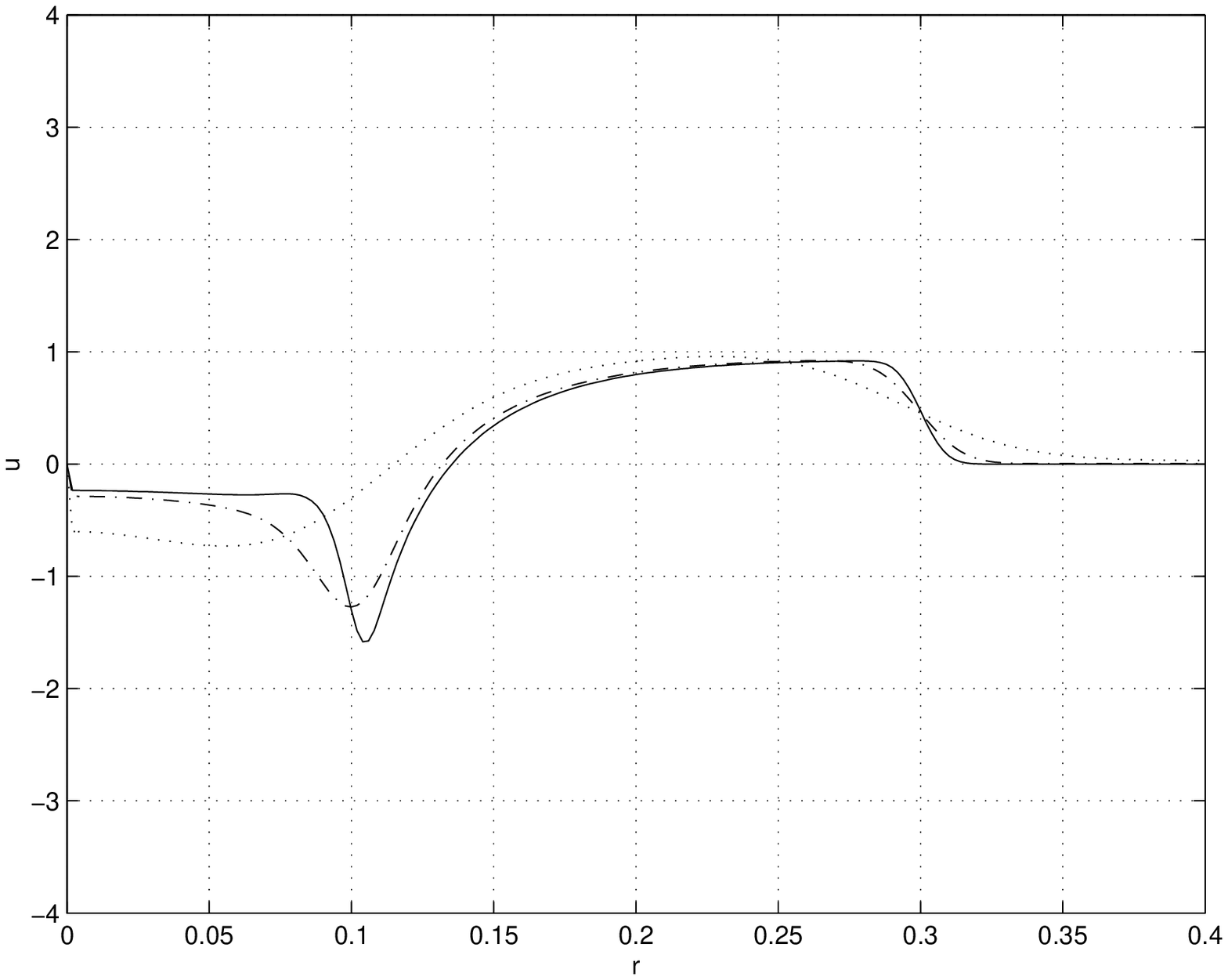} &
\includegraphics[width=0.5\textwidth]{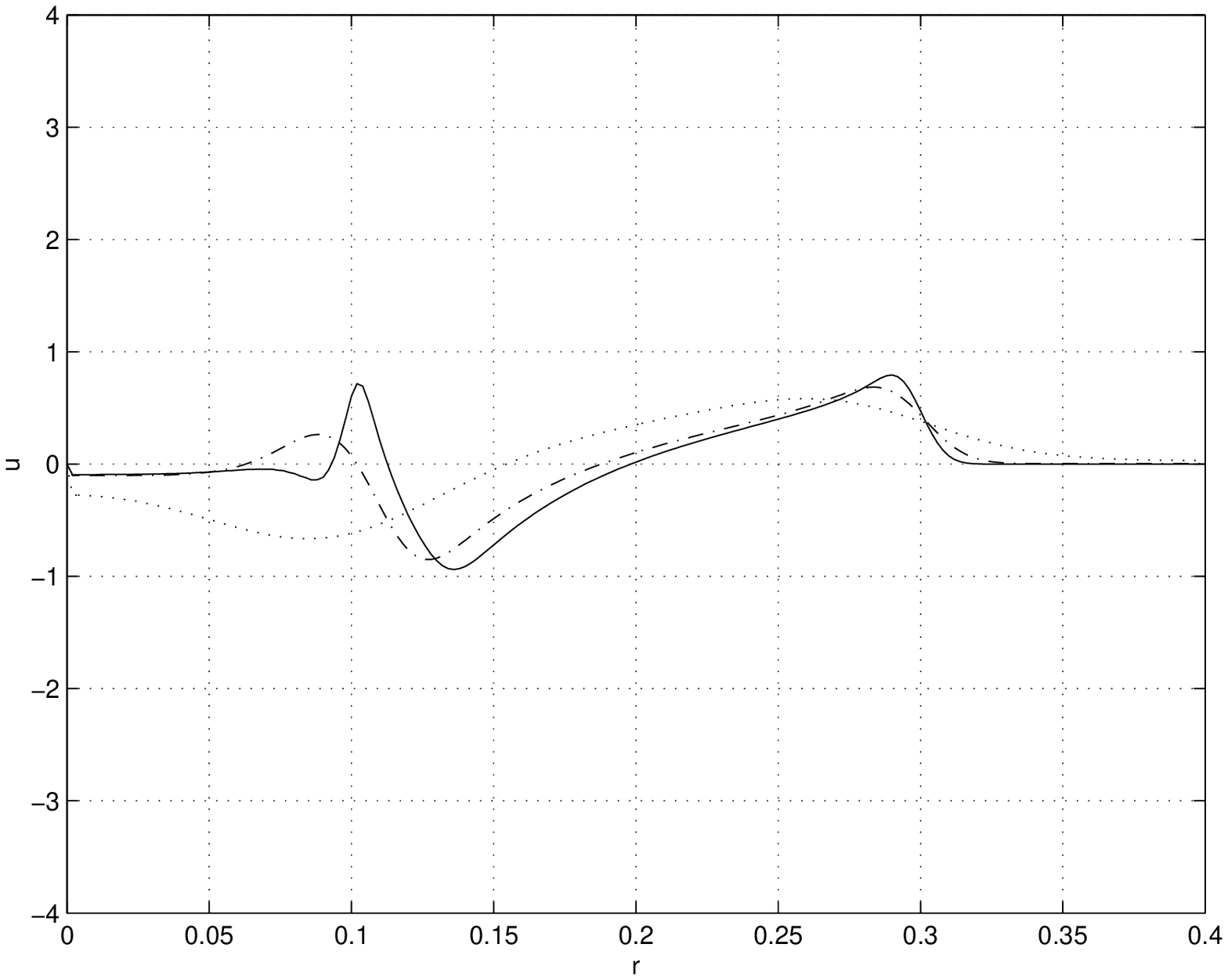} \\
\scriptsize{$G ^\prime (u) = u ^9$} &
\scriptsize{$G ^\prime (u) = \sinh (5 u) - 5 u$} \\
\includegraphics[width=0.5\textwidth]{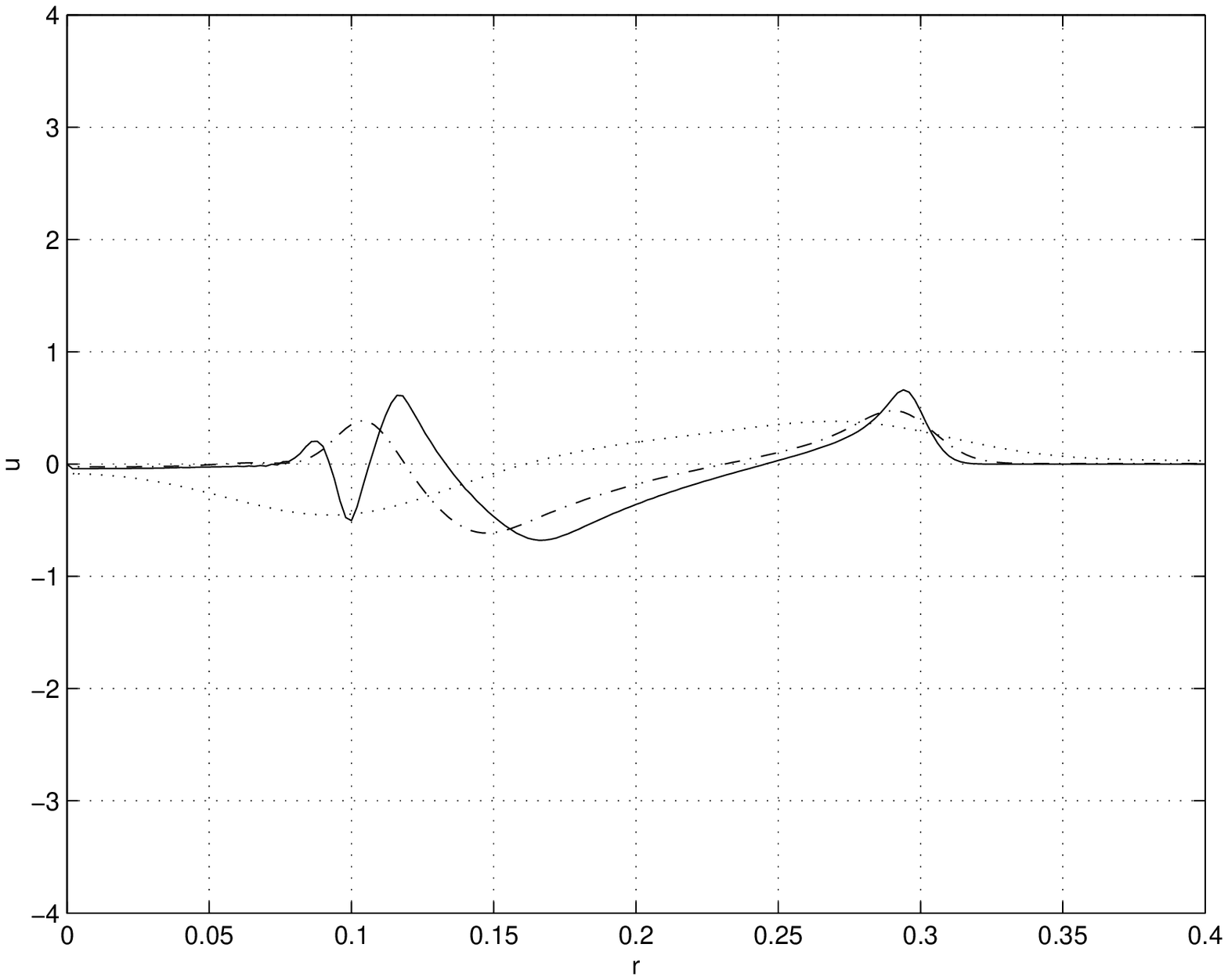} &
\includegraphics[width=0.5\textwidth]{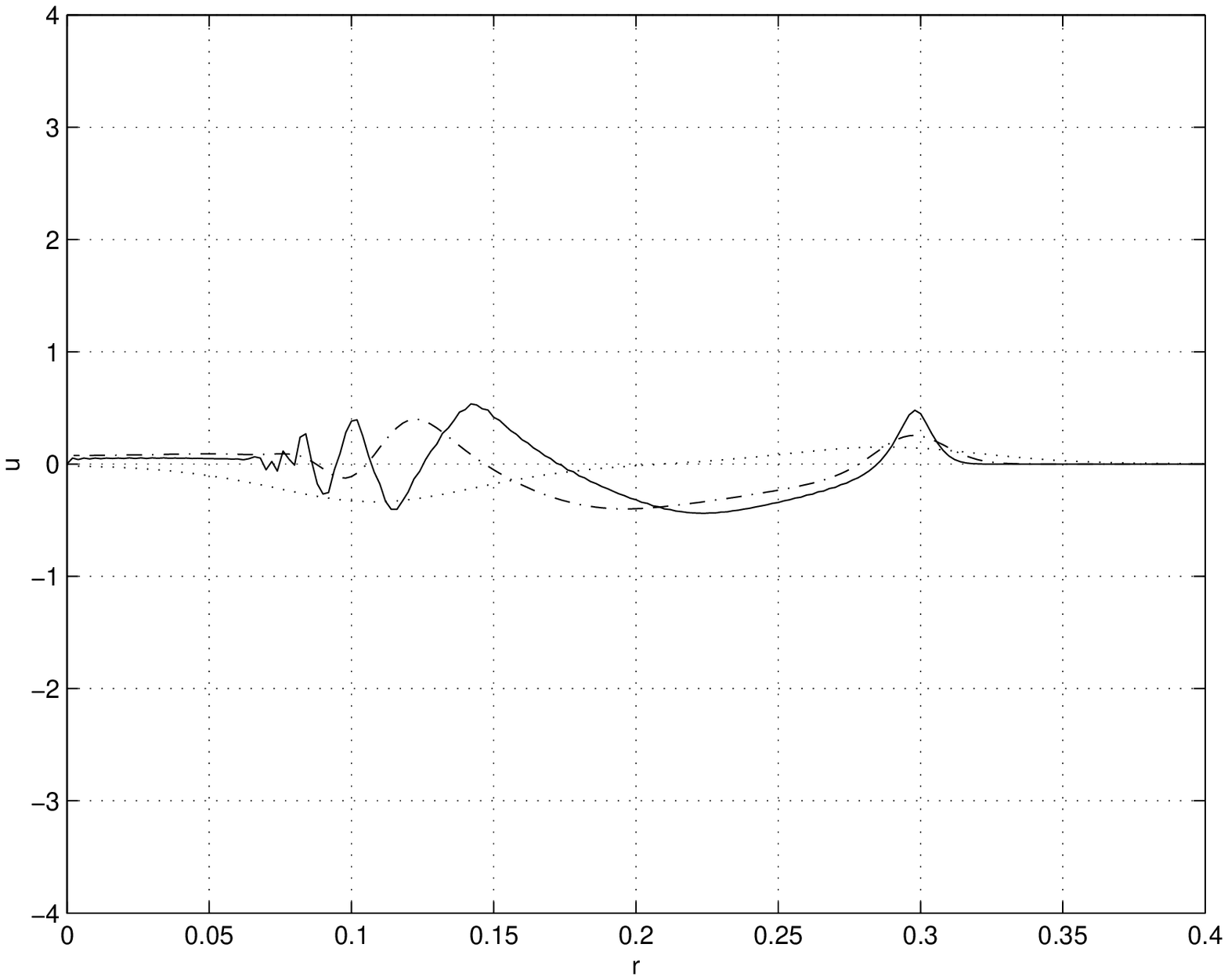} \\
\end{tabular}}
\caption{Approximate radial solutions of (\ref{paperproblem}) at
$t = 0.2$, for $\gamma = 5$, initial data $\phi (r) = 0$ and $\psi
(r) = 100 h (r)$, and $\beta = 0$ (solid), $\beta = 0.0005$
(dashed) and $\beta = 0.005$ (dotted). \label{Fig4-3-1}} %
\end{figure}

\begin{figure}[tcb]
\begin{tabular}{cc}
\scriptsize{$G ^\prime (u) = u ^3$} &
\scriptsize{$G ^\prime (u) = u ^3$} \\
\includegraphics[width=0.45\textwidth]{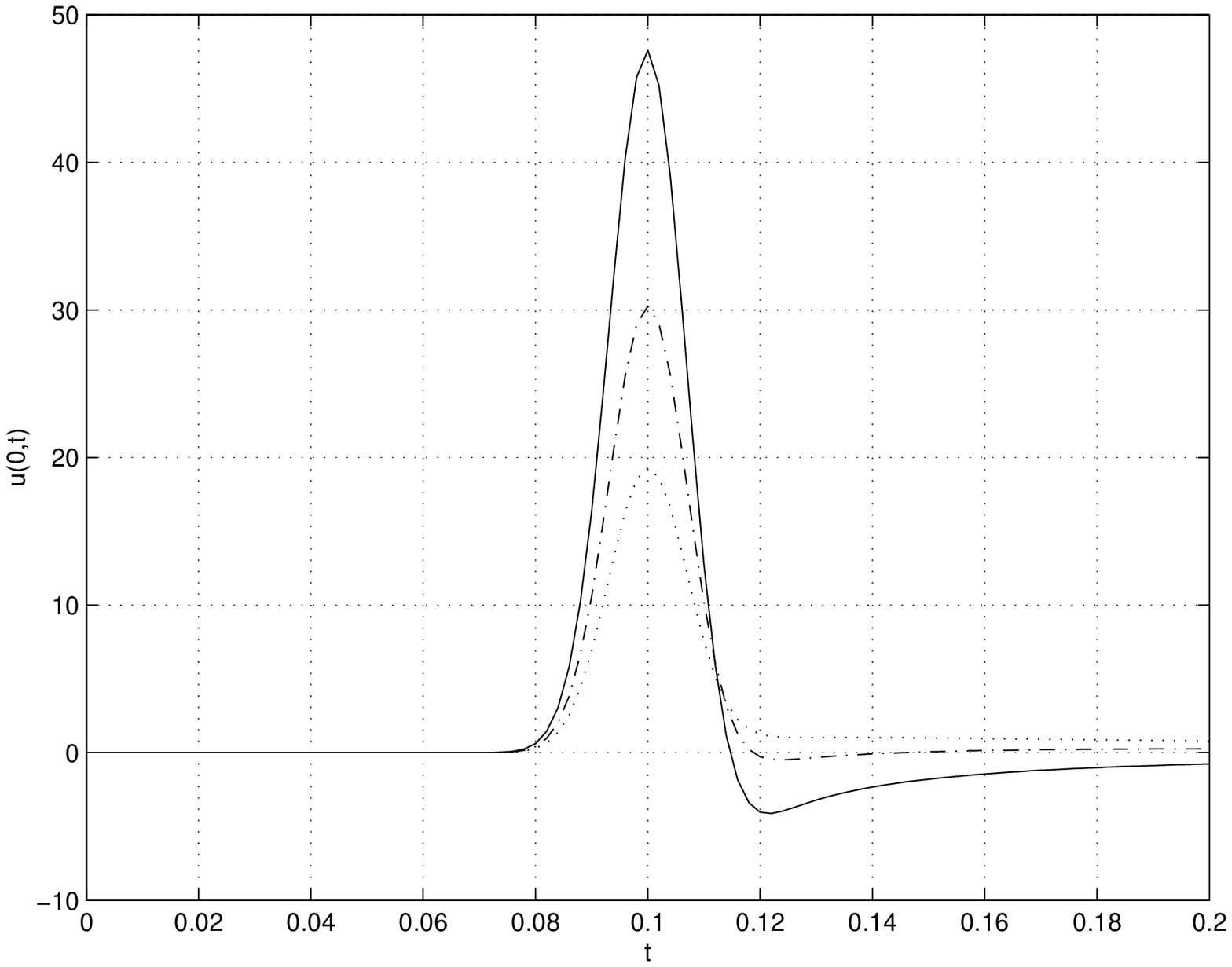} &
\includegraphics[width=0.45\textwidth]{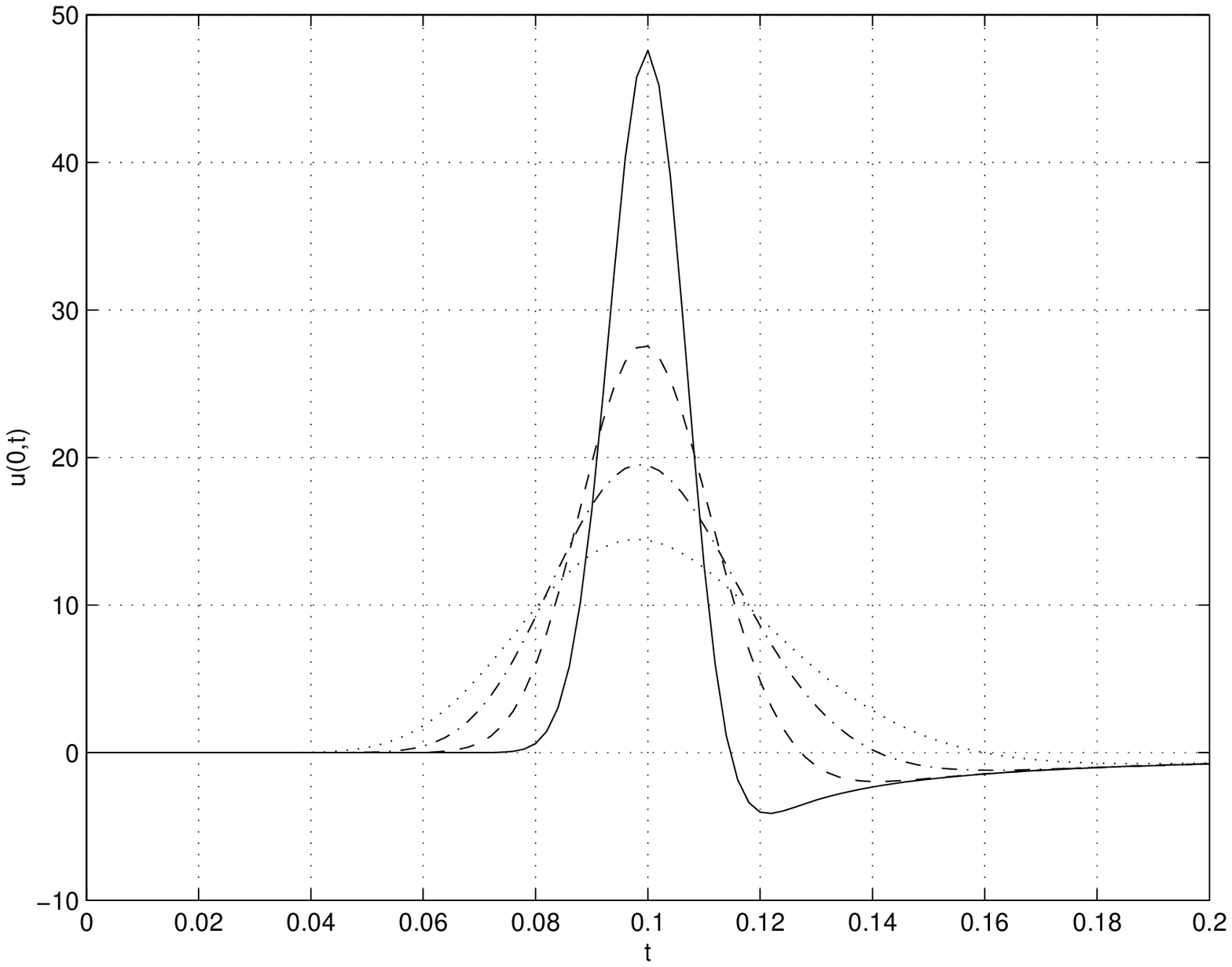} \\
\scriptsize{$G ^\prime (u) = u ^5$} &
\scriptsize{$G ^\prime (u) = u ^5$} \\
\includegraphics[width=0.45\textwidth]{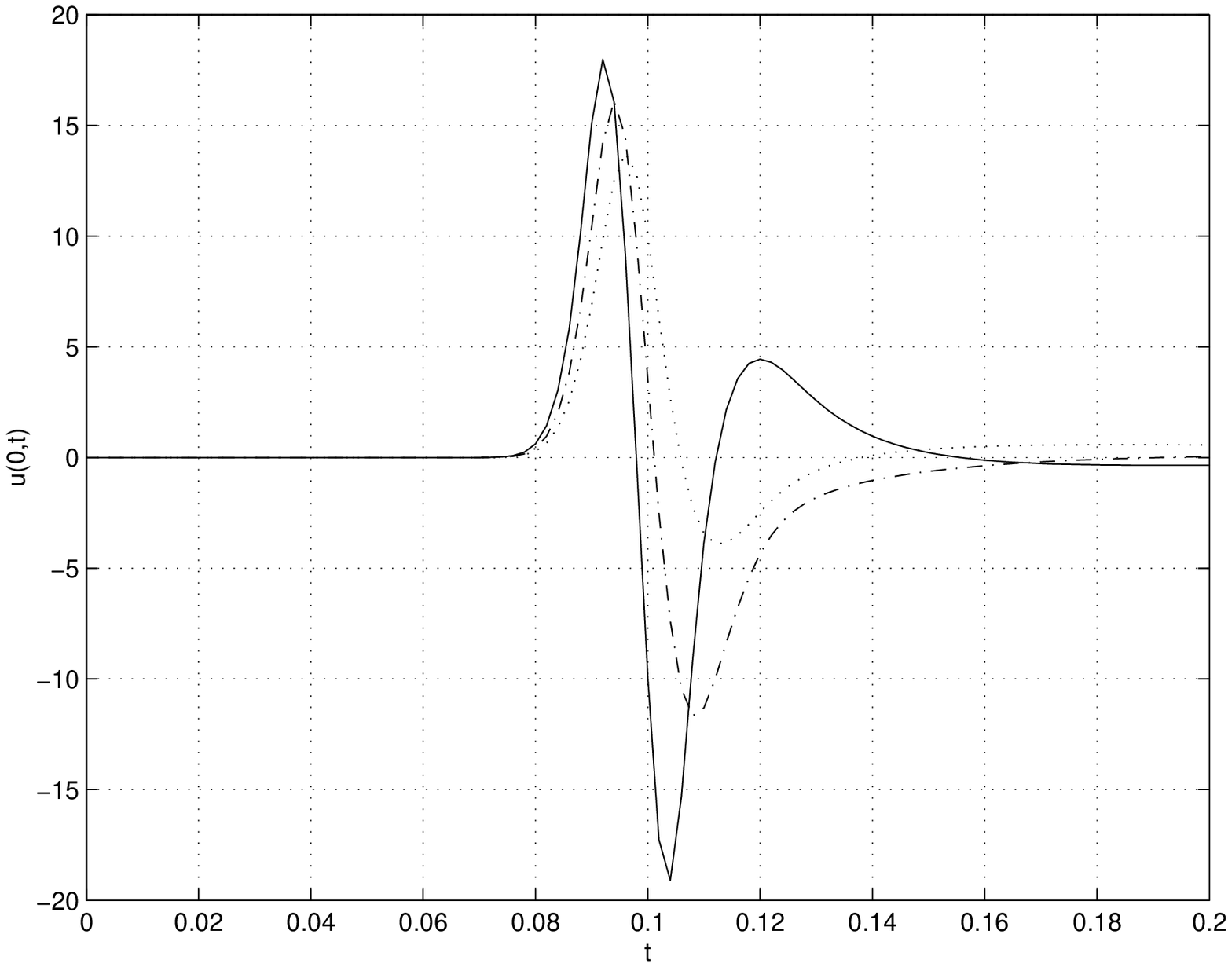} &
\includegraphics[width=0.45\textwidth]{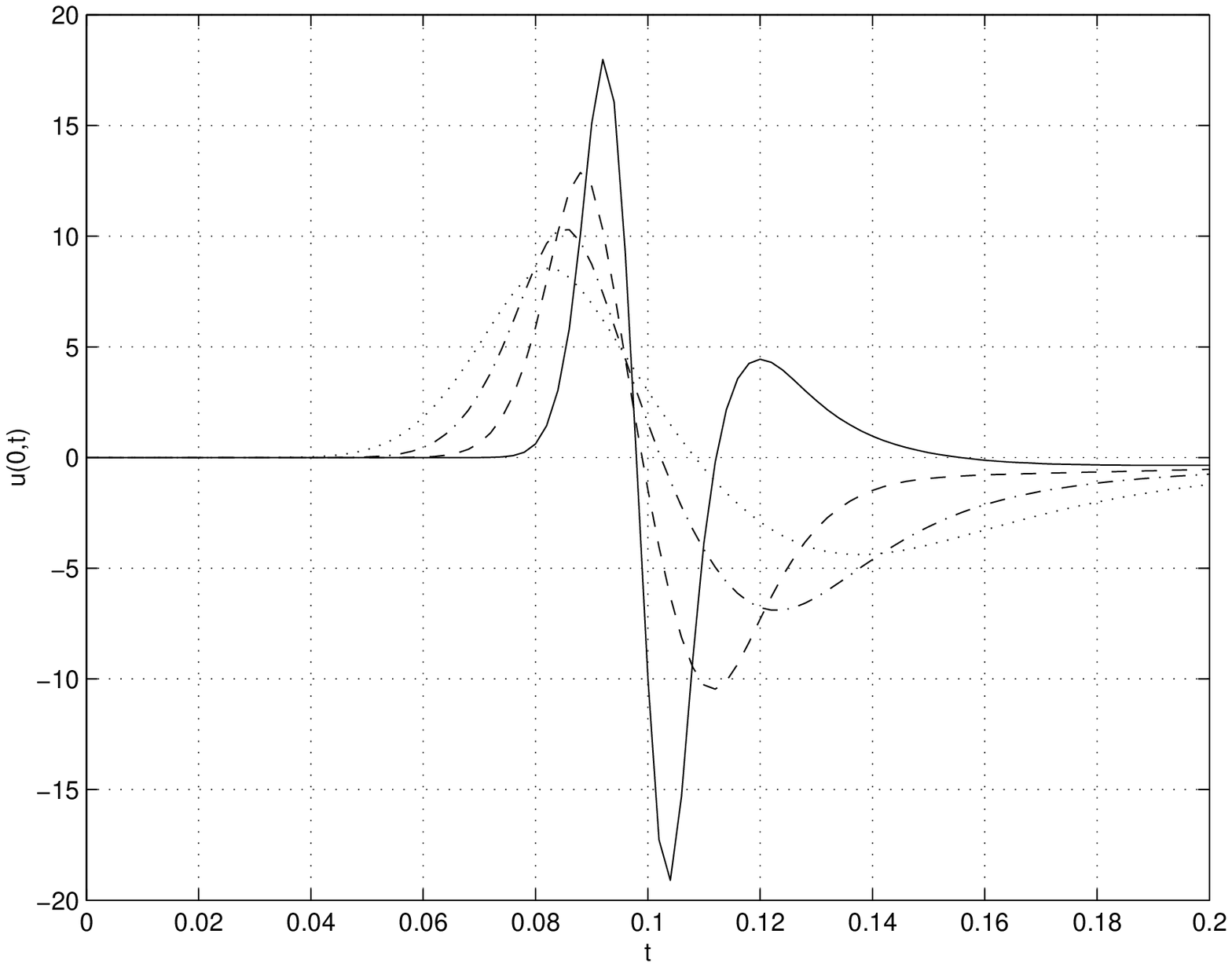} \\
\scriptsize{$G ^\prime (u) = u ^7$} &
\scriptsize{$G ^\prime (u) = u ^7$} \\
\includegraphics[width=0.45\textwidth]{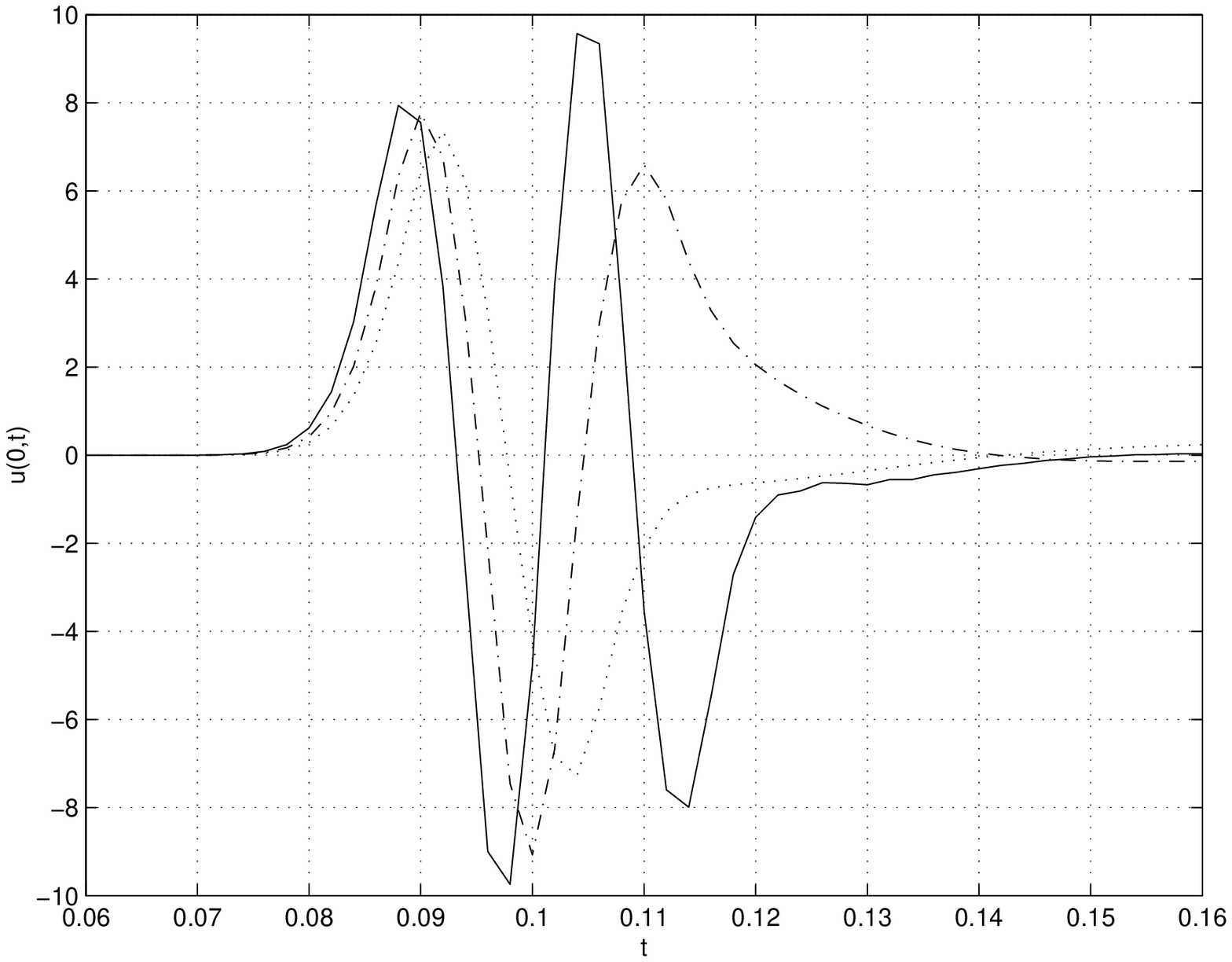} &
\includegraphics[width=0.45\textwidth]{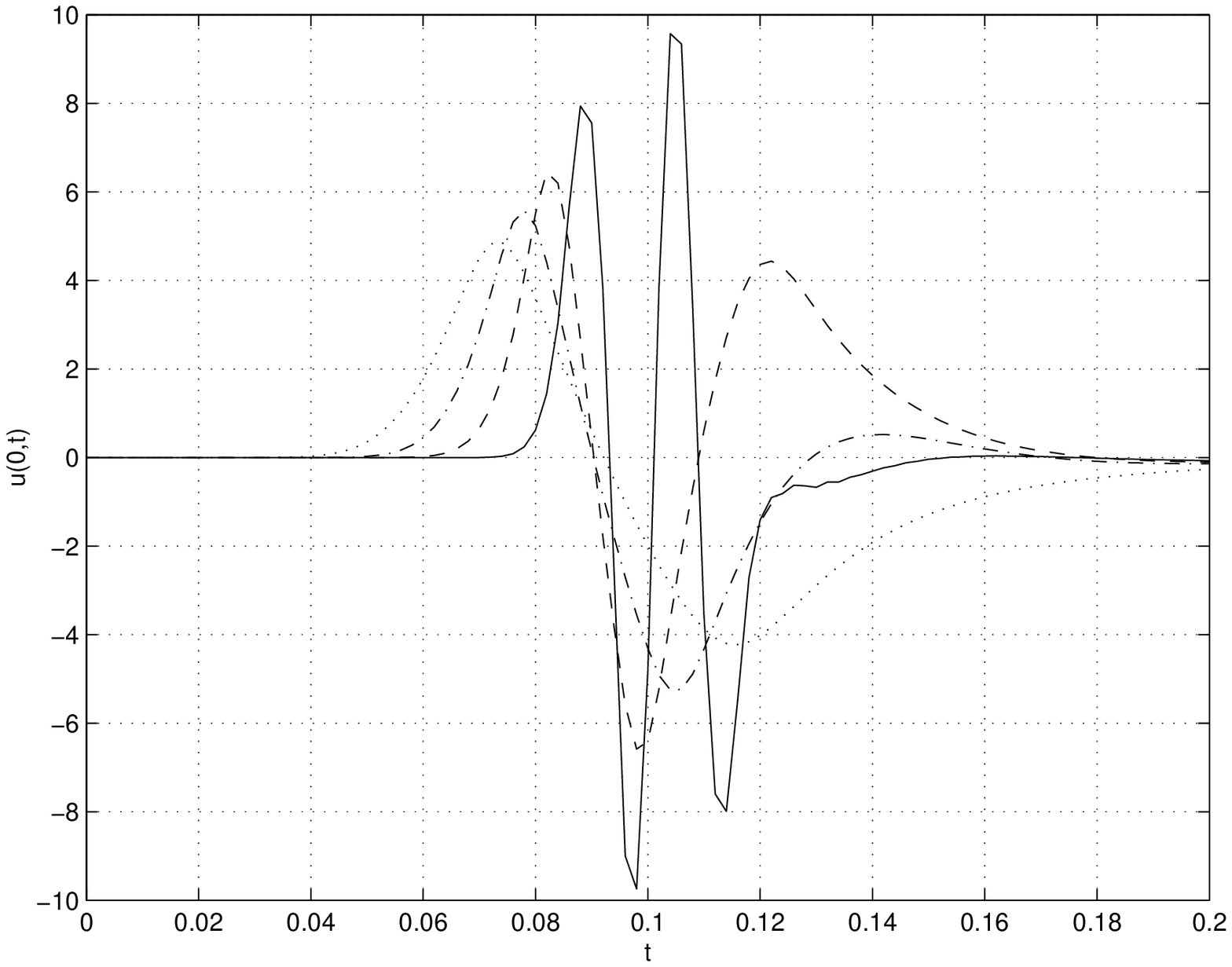} \\
\end{tabular}
\caption{Approximate value of solutions to (\ref{paperproblem})
near the origin vs. time for different nonlinear terms, and
initial conditions $\phi (r) = 0 $ and $\psi (r) = 100 h (r)$.
Left column: $\beta = 0$ and $\gamma = 0$ (solid), $\gamma = 10$ 
(dashed) and $\gamma = 20$ (dotted); right column: $\gamma =0$ and
$\beta = 0$ (solid), $\beta = 0.001$ (dashed), $\beta = 0.0025$
(dashed-dotted) and $\beta = 0.005$ (dotted). \label{Fig4-31}}
\end{figure}

\begin{figure}[tcb]
\centerline{
\begin{tabular}{cc}
\scriptsize{$G ^\prime (u) = u ^3$} &
\scriptsize{$G ^\prime (u) = u ^3$} \\
\includegraphics[width=0.45\textwidth]{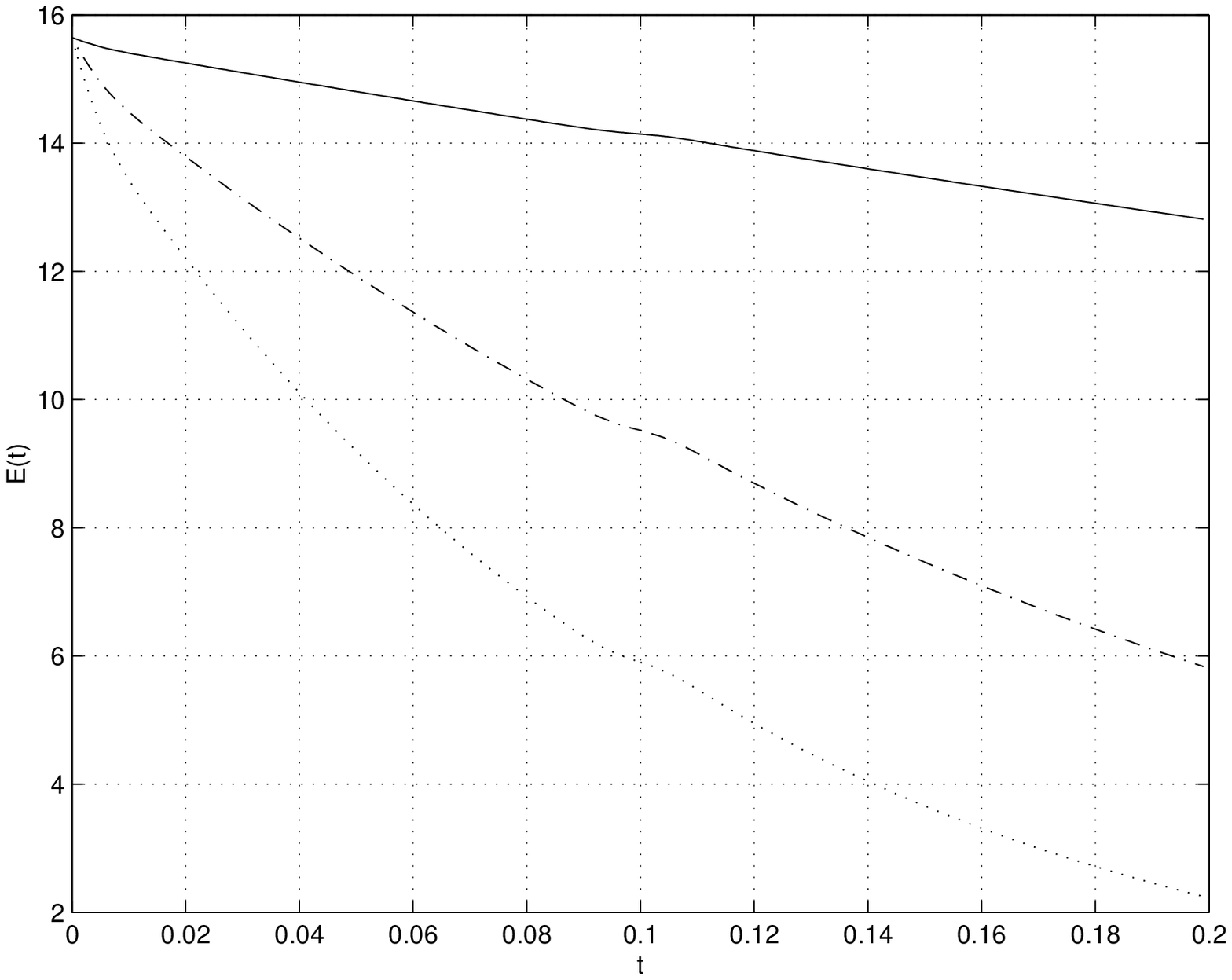} &
\includegraphics[width=0.45\textwidth]{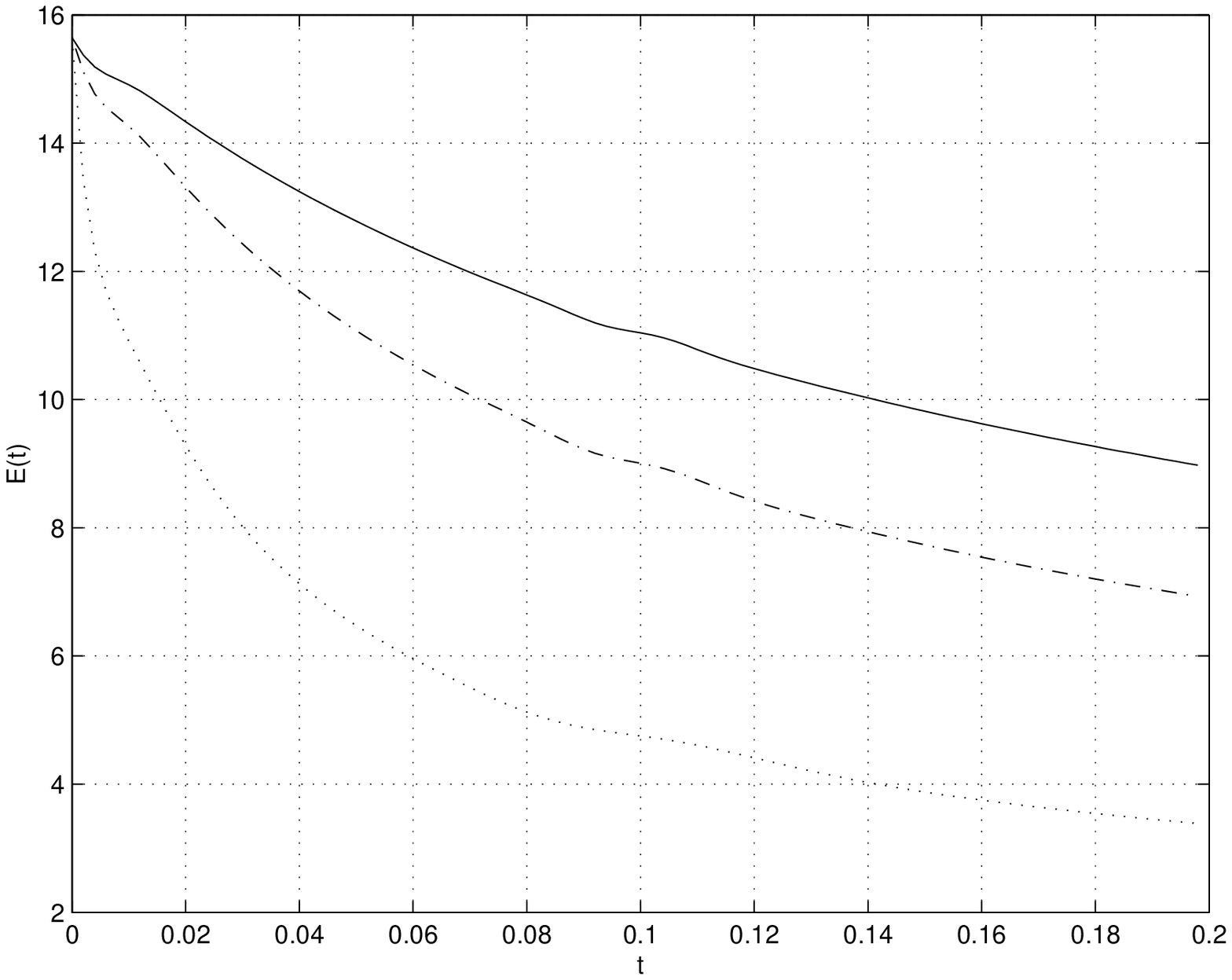} \\
\scriptsize{$G ^\prime (u) = u ^5$} &
\scriptsize{$G ^\prime (u) = u ^5$} \\
\includegraphics[width=0.45\textwidth]{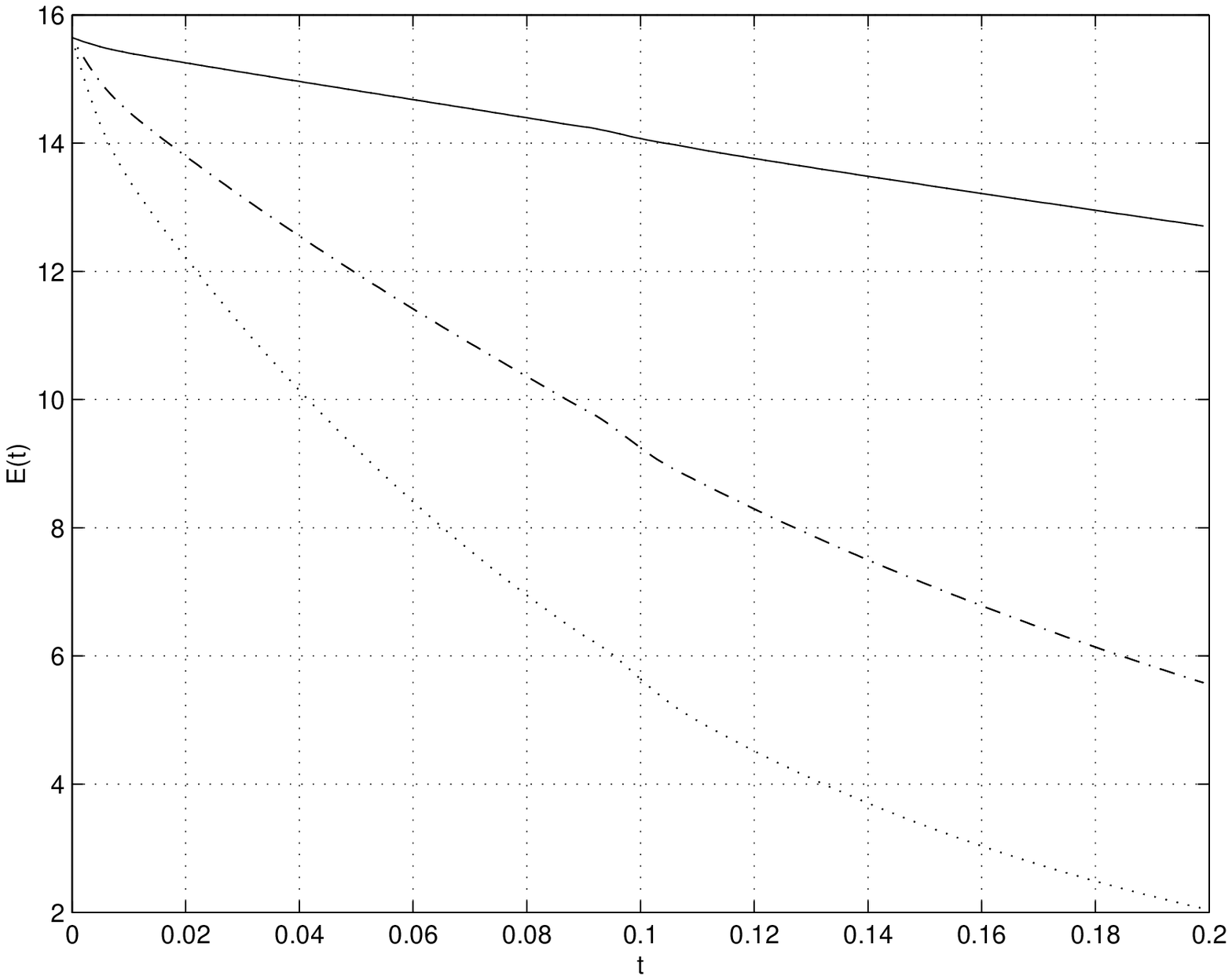} &
\includegraphics[width=0.45\textwidth]{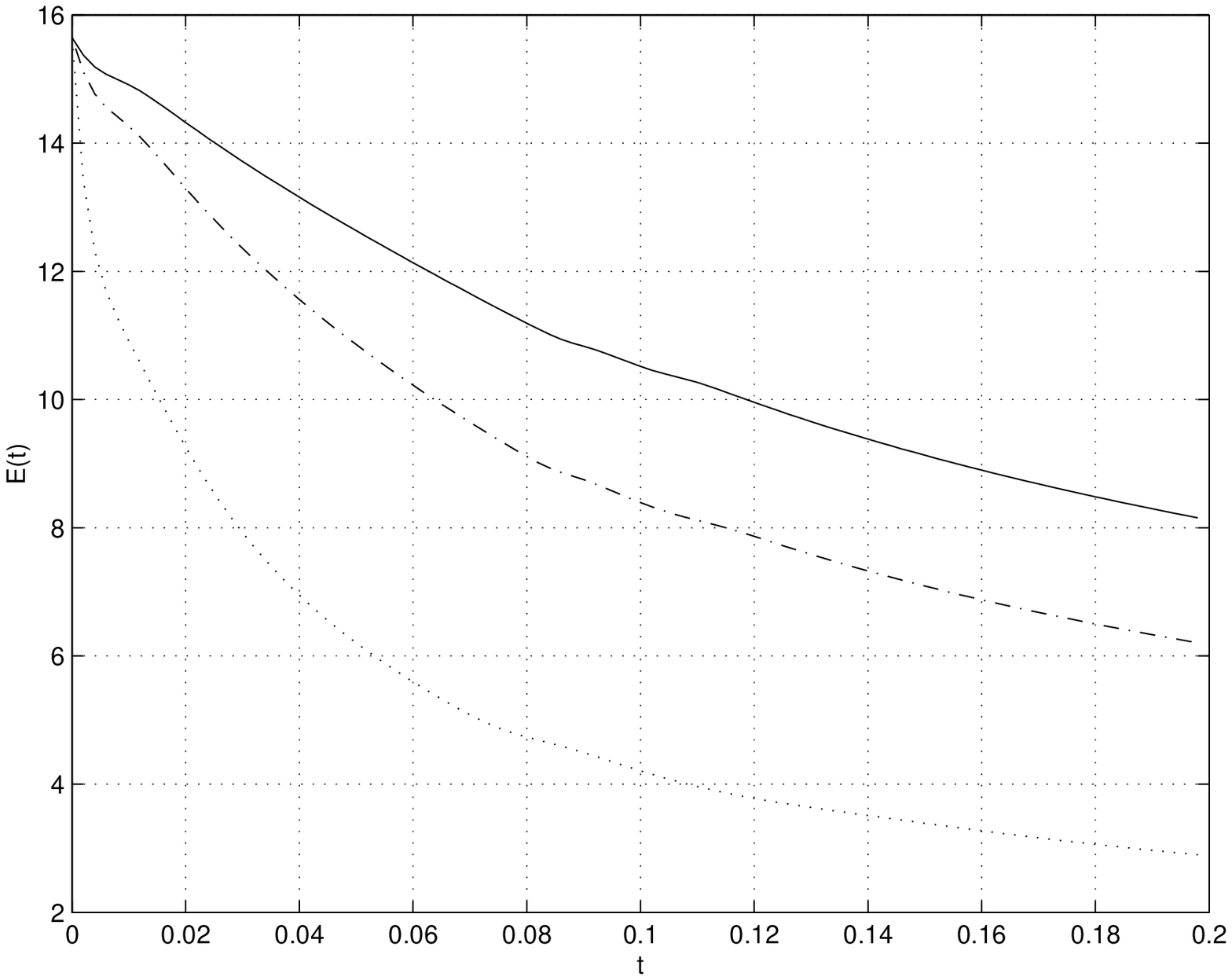} \\
\scriptsize{$G ^\prime (u) = u ^7$} &
\scriptsize{$G ^\prime (u) = u ^7$} \\
\includegraphics[width=0.45\textwidth]{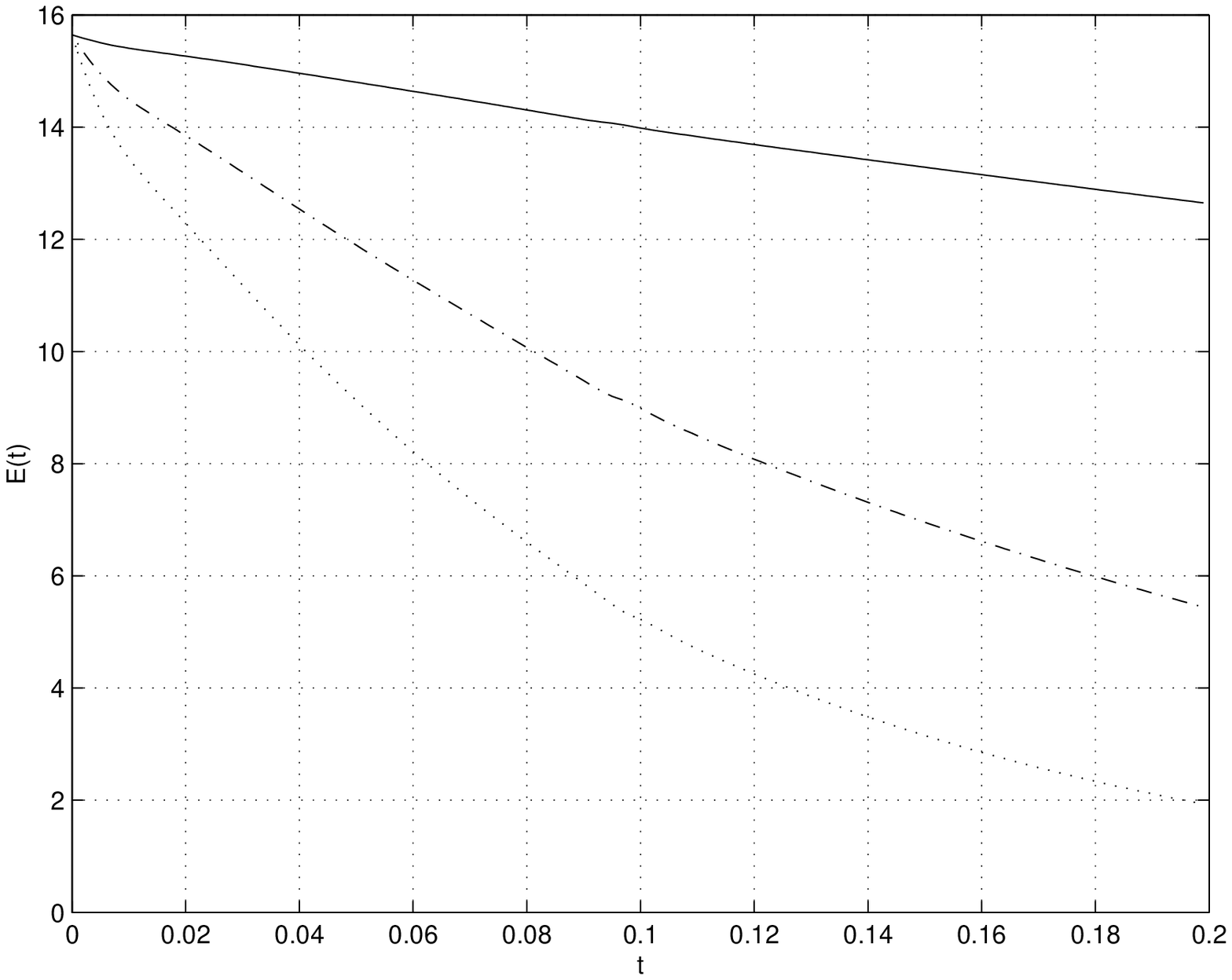} &
\includegraphics[width=0.45\textwidth]{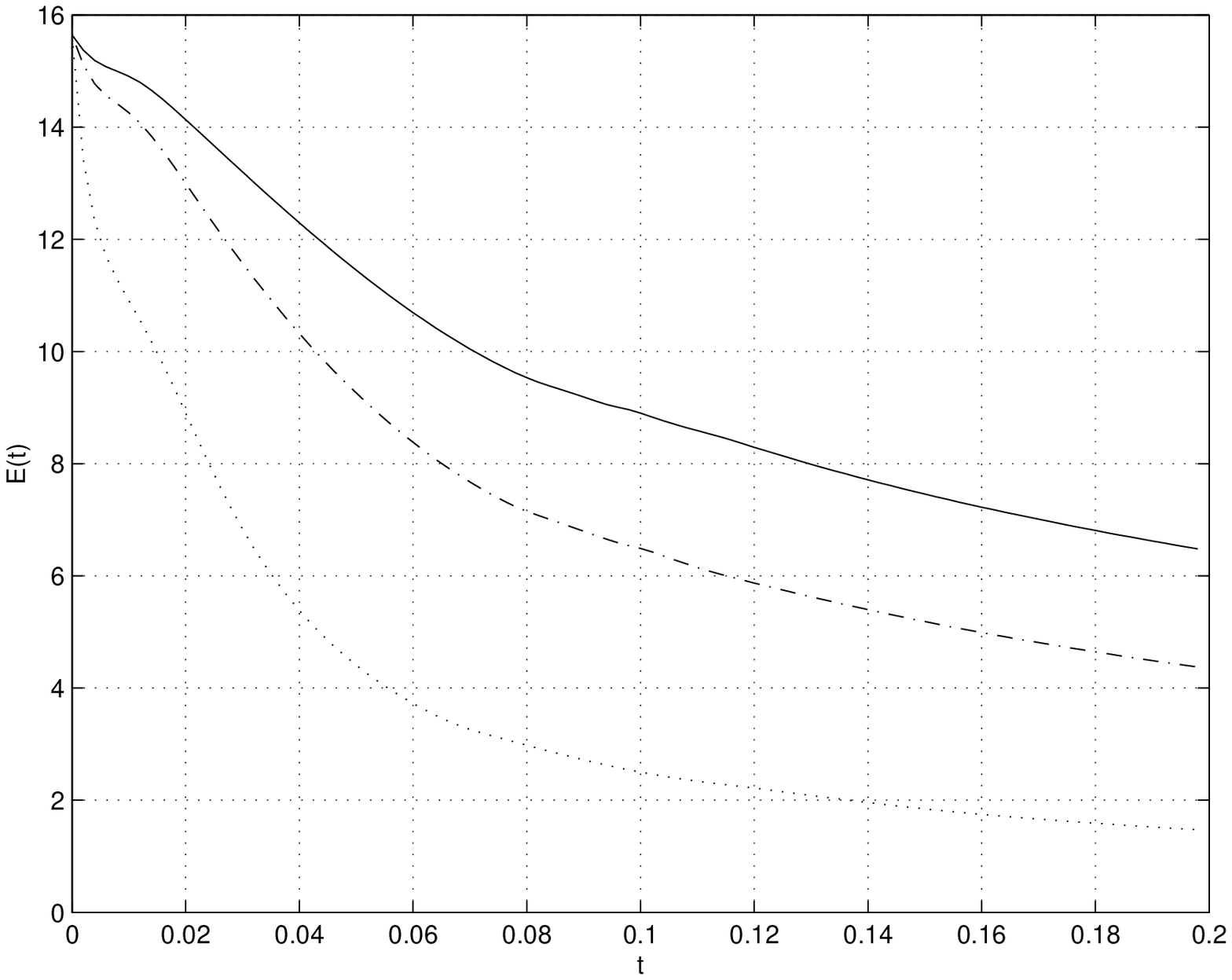} \\
\end{tabular}}
\caption{Total energy vs. time for initial data $\phi (r) = 0$ and
$\psi (r) = 100 h (r)$. Left column: $\beta = 0$ and $\gamma = 1$
(solid), $5$ (dashed), $10$ (dotted). Right column: $\gamma = 0$
and $\beta = 0.0005$ (solid), $0.001$ (dashed), $0.005$ (dotted).
\label{Fig4-3-2}} %
\end{figure}

Assume that $G \colon \mathbb {R} \rightarrow \mathbb {R}$ is
continuously differentiable, and that $w ( \bar {x} , t )$ is a
solution of (\ref{paperproblem}) in a domain $D$ of $\mathbb {R}
^3$. Moreover, we assume that $\nabla w \cdot \hat {\mathrm {n}}$
is zero near the boundary of $D$ at all time, where $\hat {\mathrm
{n}}$ denotes the unit vector normal to the boundary of $D$. The
Lagrangian associated with our nonlinear modified Klein-Gordon
equation is given by
\begin{equation}
\mathcal {L}  = \frac {1} {2} \left\{ \left( \frac {\partial w}
{\partial t} \right) ^2 - | \nabla w | ^2 - m ^2 w ^2 \right\} -
G(w). \nonumber
\end{equation}
It is easy to derive the following expression for the total energy
associated with our nonlinear dissipative Klein-Gordon-like
equation:
\begin{equation}
E (t) = \iiint _D \left\{ \frac {1} {2} \left( \frac {\partial w}
{\partial t} \right) ^2 + \frac {1} {2} | \nabla w | ^2 + \frac {m
^2} {2} w ^2 + G (w) \right\} d \bar {x}. \label{EnergyEq1}
\end{equation}

\begin{proposition}
The instantaneous rate of change with respect to time of the total
energy associated with the PDE in {\rm (\ref{paperproblem})} is
given by
\begin{equation}
E ^\prime (t) =  - \iiint _D \left\{ \beta \left\Vert \nabla
\left( \frac {\partial w} {\partial t} \right) \right\Vert ^2 +
\gamma \left( \frac {\partial w} {\partial t} \right) ^2 \right\}
\ d \bar {x}. \nonumber
\end{equation}
\end{proposition}

\begin{proof} Taking derivative on both sides of Equation
(\ref{EnergyEq1}), we obtain that
\begin{eqnarray}
\frac {d E} {d t} & = & \iiint _D \frac {\partial w} {\partial t}
\left\{ \frac {\partial ^2 w} {\partial t^2} + m ^2 w + G ^\prime
(w) \right\} d \bar {x} + \frac {1} {2} \iiint _D \frac {\partial}
{\partial t} | \nabla w | ^2 d \bar {x} \nonumber \\
 & = & \iiint _D \frac {\partial w} {\partial t} \left\{ \frac
{\partial ^2 w} {\partial t^2} - \nabla ^2 w + m ^2 w + G ^\prime
(w) \right\} d \bar {x} + \iint _{\partial D} \frac {\partial w}
{\partial t} \nabla w \cdot \hat {\mathrm {n}}
\ d \sigma \nonumber \\
 & = & \beta \iiint _D \frac {\partial w} {\partial t} \nabla ^2
\left( \frac {\partial w} {\partial t} \right) \ d \bar {x} -
\gamma \iiint _D \left( \frac {\partial w} {\partial t} \right) ^2
d \bar {x} + \iint _{\partial D} \frac {\partial w} {\partial t}
\nabla w \cdot \hat {\mathrm {n}} \ d \sigma. \nonumber
\end{eqnarray}

On the other hand, from Green's first identity we see that
\begin{equation}
\iiint _D \frac {\partial w} {\partial t} \nabla ^2 \left( \frac
{\partial w} {\partial t} \right) \ d \bar {x} = \iint _{\partial
D} \frac {\partial w} {\partial t} \frac {\partial} {\partial t}
\left( \nabla w \cdot \hat {\mathrm {n}} \right) \ d \sigma -
\iiint _D \left\Vert \nabla \left( \frac {\partial w} {\partial t}
\right) \right\Vert ^2 \ d \bar {x}. \nonumber
\end{equation}
The surface integrals in these last two equations are equal to
zero, whence the result follows.
\end{proof}

It is worthwhile noticing that if $\beta$ and $\gamma$ are
positive then the total energy is decreasing in time. Also, if
$\beta$ and $\gamma$ are both equal to zero then the energy is
conserved. Finally, if $\beta$ is zero then the expression of
$E ^\prime (t)$ coincides with the one derived in
\cite{StraussBook}.

\smallskip
Let us assume now that $G$ is nonnegative. The total energy in
this case is likewise nonnegative and the integral of every term
in (\ref{EnergyEq1}) is bounded by $\sqrt {2 E (t)} / m$. In particular, this
last statement implies that the integral of $w ^2$ at time $t$ is
bounded by $E (t)$. For those times $t$ for which $E (t)$ is
finite (and particularly for the case when $\beta$ and $\gamma$
are both equal to zero), this means that $w$ is a
square-integrable function in the first variable at time $t$.

\smallskip
Let $G ^\prime (w) = w ^p$ with $p > 1$. Assuming that $w$ is a
radially symmetric solution of the damped nonlinear Klein-Gordon
equation in a sphere $D$ with center in the origin and radius $L$,
and using the transformation $v (r , t) = r w (r , t)$ the energy
expression adopts the form $E(t) = 4 \pi E _0 (t)$, with
\begin{equation}
E _0 (t) = \int _0 ^L \left\{ \frac {1} {2} \left( \frac {\partial
v} {\partial t} \right) ^2 + \frac {1} {2} \left( \frac {\partial
v} {\partial r} \right) ^2 + \frac {m ^2} {2} v ^2 + r ^{1 - p} G
(v) \right\} dr. \label{EnergyChichona}
\end{equation}

The instantaneous rate of change of energy is given by $E ^\prime
(t) = 4 \pi E _0 ^\prime (t)$, where
\begin{equation}
E ^\prime _0 (t) = -  \int _0 ^L \left\{ \beta \left( \frac
{\partial ^2 v} {\partial r \partial t} - \frac {1} {r} \frac
{\partial v} {\partial t} \right) ^2 + \gamma \left( \frac
{\partial v} {\partial t} \right) ^2 \right\} dr. \nonumber
\end{equation}
It is possible to reproduce now the argument in
\cite{StraussVazquez} to show that for every $t$ and nonzero $r$,
$| w (r , t) | \leq \sqrt {2 E _0 (t) } / r$. This means in
particular that if a solution were unbounded, it would have to be
unbounded at the origin.

\smallskip
The discrete energy is given by
\begin{eqnarray}
\frac {E_0 ^n} {\Delta r} & = & \frac {1} {2} \sum _{j=0} ^{m-1}
\left(
\frac {v _j ^{n+1} - v _j ^n } {\Delta t} \right) ^2 + %
\frac {1} {2} \sum _{j=0} ^{m-1} \left( \frac {v _{j+1} ^{n+1} - v
_j ^{n+1}} {\Delta r}\right) \left( \frac {v _{j+1} ^n - v _j ^n}
{\Delta r} \right) \nonumber \\ %
 & & \qquad + \frac {1} {2} \sum _{j=0} ^{m-1} \frac {(v _j ^{n+1}) ^2 +
 (v _j ^n) ^2} {2} + \sum _{j=1} ^{m-1} \frac {G(v _j ^{n+1}) +
 G(v _j ^n)} {2 (j \Delta r) ^{p-1}}. \nonumber
\end{eqnarray}
This expression is obviously consistent with
(\ref{EnergyChichona}). Moreover, taking the difference between $E
_0 ^n / \Delta r$ and $E _0 ^{n - 1} / \Delta r$ and simplifying
after using (\ref{EasyScheme1}), it can be shown that
\begin{eqnarray}
\frac {E _0 ^n - E _0 ^{n - 1}} {\Delta t} & = & - \beta \sum _{j
= 1} ^{m - 1} \left( \frac { v _j ^{n + 1} - v _j ^{n - 1} } {2
\Delta t} \right) \left( \frac { ( v _j ^{n + 1} - v _j ^{n - 1} )
- ( v _{j - 1} ^{n + 1} - v _{j - 1} ^{n - 1} ) } { \Delta t (
\Delta r ) ^2} \right) \Delta r \nonumber \\
 & & \qquad - \gamma \sum _{j
= 1} ^{m - 1} \left( \frac {v _j ^{n + 1} - v _j ^{n - 1}} {2
\Delta t} \right) ^2 \Delta r. \nonumber
\end{eqnarray}

For $\beta = 0$ this expression provides us with a consistent
approximation to the instantaneous rate of change of energy.
Numerical results demonstrate that energy decreases as a function
of time for $\beta > 0$, which is in agreement with the
corresponding instantaneous change of energy as a function of
time.

\section{Numerical results}
\label{sec3}

\begin{table}[tcb]
\caption{Relative differences of externally damped solutions to
(\ref{paperproblem}) with respect to the corresponding undamped
solution at different time steps. \label{table4-1}}
\begin{tabular*}{\textwidth}{@{\extracolsep{\fill}}cccccc} \hline
{\bf Time step} & \multicolumn{5}{c}{{\bf Relative differences}}
\\ \cline{2-6} %
$n$ & $\gamma = 0.1$ & $\gamma = 0.5$ & $\gamma = 1$ & $\gamma =
5$ & $\gamma = 10$  \\ \hline %
0   & 0.0000 & 0.0000 & 0.0000 & 0.0000 & 0.0000 \\ %
20  & 0.0028 & 0.0142 & 0.0283 & 0.1395 & 0.2693 \\ %
40  & 0.0103 & 0.0509 & 0.1006 & 0.4491 & 0.7706 \\ %
60  & 0.0167 & 0.0821 & 0.1611 & 0.6579 & 0.9573 \\ %
80  & 0.0192 & 0.0942 & 0.1836 & 0.6954 & 0.9387 \\ %
100 & 0.0200 & 0.0977 & 0.1896 & 0.6994 & 0.9308 \\ \hline %
\end{tabular*}
\end{table}

\begin{table}[tcb]
\caption{Table of relative differences of externally damped solutions of
(\ref{paperproblem}) with respect to the corresponding undamped
solution at $t = 0.2$. \label{table4-2}}
\begin{tabular*}{\textwidth}{@{\extracolsep{\fill}}cccccc} \hline
{\bf Nonlinear Term} & \multicolumn{5}{c}{{\bf Relative
differences}} \\ \cline{2-6} %
$G ^\prime (u) $ & $\gamma = 0.1$ & $\gamma = 0.5$ & $\gamma = 1$
& $\gamma =5$ & $\gamma = 10$  \\ \hline %
$0$                 & 0.0098 & 0.0478 & 0.0923 & 0.3642 & 0.5631 \\ %
$u ^3$        & 0.0097 & 0.0477 & 0.0929 & 0.3528 & 0.5554 \\ %
$u ^5$        & 0.0137 & 0.0665 & 0.1287 & 0.4024 & 0.6418 \\ %
$u ^7$        & 0.0171 & 0.0833 & 0.1618 & 0.5068 & 0.7819 \\ %
$u ^9$        & 0.0204 & 0.0999 & 0.1728 & 0.5736 & 0.8488 \\ %
$\sinh (5 u) - 5 u$ & 0.0263 & 0.1377 & 0.2518 & 0.6284 & 0.8813 \\ \hline %
\end{tabular*}
\end{table}

\begin{table}[tcb]
\caption{Relative differences of internally damped solutions to
(\ref{paperproblem}) with respect to the corresponding undamped
solution at different time steps.  \label{table4}}
\begin{tabular*}{\textwidth}{@{\extracolsep{\fill}}cccccc} \hline
{\bf Time step} & \multicolumn{5}{c}{{\bf Relative differences}}
\\ \cline{2-6} %
$n$ & $\beta = 10 ^{-6}$ & $\beta = 10 ^{-5}$ & $\beta = 10 ^{-4}$
& $\beta = 0.0005$ & $\beta = 0.001$ \\ \hline %
0   & 0.0000 & 0.0000 & 0.0000 & 0.0000 & 0.0000 \\ %
20  & 0.0005 & 0.0054 & 0.0517 & 0.2188 & 0.3682 \\ %
40  & 0.0030 & 0.0300 & 0.2640 & 0.8281 & 1.1457 \\ %
60  & 0.0156 & 0.0996 & 0.1493 & 1.6536 & 1.1460 \\ %
80  & 0.0102 & 0.0970 & 0.7242 & 1.1530 & 1.2138 \\ %
100 & 0.0080 & 0.0772 & 0.5751 & 1.0406 & 1.1435 \\ \hline %
\end{tabular*}
\end{table}

\begin{table}[tcb]
\caption{Table of relative differences of internally damped solutions of
(\ref{paperproblem}) with respect to the corresponding undamped
solution at $t = 0.2$. \label{table5}}
\begin{tabular*}{\textwidth}{@{\extracolsep{\fill}}cccccc} \hline
{\bf Nonlinear Term} & \multicolumn{5}{c}{{\bf Relative
differences}} \\ \cline{2-6} %
$G ^\prime (u) $ & $\beta = 10 ^{-6}$ & $\beta = 10 ^{-5}$ & $\beta = 10 ^{-4}$
& $\beta = 0.0005$ & $\beta = 0.001$ \\ \hline %
$0$                 & 0.0003 & 0.0027 & 0.0242 & 0.0859 & 0.1326 \\ %
$u ^3$        & 0.0003 & 0.0032 & 0.0289 & 0.1040 & 0.1621 \\ %
$u ^5$        & 0.0011 & 0.0105 & 0.0948 & 0.3374 & 0.5043 \\ %
$u ^7$        & 0.0023 & 0.0224 & 0.1825 & 0.5663 & 0.7327 \\ %
$u ^9$        & 0.0041 & 0.0397 & 0.3133 & 0.7318 & 0.9256 \\ %
$\sinh (5 u) - 5 u$ & 0.0063 & 0.0577 & 0.4717 & 0.9403 & 1.1007 \\ \hline %
\end{tabular*}
\end{table}

The numerical results presented in this section correspond to approximate solutions
of the dissipative, nonlinear, modified Klein-Gordon equation
\begin{equation*}
\frac {\partial ^2 u} {\partial t ^2} - \nabla ^2 u
    - \beta \frac {\partial} {\partial t} \left( \nabla ^2 u \right)
    + \gamma \frac {\partial u} {\partial t} + u + G ^\prime (u)
    = 0,
\end{equation*}
obtained using a tolerance of $10 ^{-5}$  and a maximum number of $20$ 
iterations on every application of Newton's method. The space and time
steps are always fixed as $\Delta r = \Delta t = 0.002$.

\subsection*{External damping}

Throughout this section we fix $\beta = 0$. 

\smallskip
Let us start considering the problem of approximating radially
symmetric solutions of (\ref{paperproblem}) with $G ^\prime (u) =
u ^7$, and initial data
$\phi (r) = h (r)$ and $\psi (r) = h ^\prime (r) + h (r) / r$,
where
\begin{equation}
h (r) = \left\{ %
\begin{array}{ll}
\displaystyle {5 \exp \left\{ 100 \left[ 1 - \frac {1} {1 - (10 r
- 1) ^2} \right] \right\} ,} & {\rm if} \ 0 \leq r < 0.2, \\
0, & {\rm if} \ 0.2 \leq r \leq 0.4.
\end{array} \right. \nonumber
\end{equation}

\smallskip
We have plotted numerical solutions of this problem for several
values of $\gamma$. The graphical results are presented in Figure
\ref{Fig4-2} for $\gamma = 0, 5, 10$. We observe first of all that
the solutions of the damped nonlinear Klein-Gordon-like equation
corresponding to small values of $\gamma$ are consistently similar
to those of the undamped case. To verify this claim
quantitatively, we consider the approximations $\bar {v} _0 ^n$
and $\bar {v} _\gamma ^n$ to the undamped and damped cases,
respectively, and compute the relative difference in the $\ell _{2
, \Delta x}$-norm via
\begin{equation}
\delta ( \bar {v} _\gamma ^n , \bar {v} _0 ^n ) = \frac {|| \bar
{v} _\gamma ^n - \bar {v} _0 ^n || _{2 , \Delta x}} { || \bar {v}
_0 ^n || _{2 , \Delta x} } \nonumber
\end{equation}
(here we follow \cite{Thomas}). The relative differences for
several values of $\gamma$ at consecutive time steps are shown in
Table \ref{table4-1}. We observe that the difference between the
solutions of the nonlinear Klein-Gordon-like equation with damping
coefficient $\gamma$ and the corresponding undamped equation can
be made arbitrarily small by taking $\gamma$ sufficiently close to
$0$.

\smallskip
We wish to corroborate this pattern for different
nonlinear terms and a different set of initial conditions.
With this objective in mind, Figure \ref{Fig4-3} depicts numerical solutions of
(\ref{paperproblem}) with $\gamma = 0, 5, 10$, nonlinear terms $G
^\prime (u) = 0, u^3, u ^5, u ^7, u ^9,$ and $\sinh (5u) - 5 u$, 
initial conditions $\phi (r) = 0$ and $\psi (r) = 100 h (r)$, and
values of $\gamma = 0, 5, 10$. More accurately, Table
\ref{table4-2} provides relative differences of these solutions at
$t = 0.2$ (for the nonlinear functions listed above and varying
values of $\gamma$) with respect to the corresponding undamped
solution, for a wider selection of values of the parameter $\gamma$.

\smallskip
It must be mentioned that, as it was expected, the total energy
was invariably decreasing for positive values of $\gamma$, and
increasing for negative values. For the value $\gamma = 0$, the
rate of change of energy is equal to zero and our numeric results
agree with \cite{StraussVazquez}. Experimental results show that
small values of $\gamma$ correspond with small values of the
discrete rate of change of the energy. This last observation
corroborates stability of our method.

\smallskip
We also observe that the amplitude of solutions corresponding to
positive values of $\gamma$ tend to decrease as time or $\gamma$
increases. Figure \ref{Fig4-3} partially corroborates that
behavior. We have computed solutions corresponding to negative
values of $\gamma$ (graphs not included in this paper) and have
verified that the amplitude of solutions increases with time and
with $| \gamma |$.

\smallskip
Finally, we have obtained graphs of the energy $E _0$ vs. time for
$G ^\prime (u) = u ^3, u ^5 , u ^7$, initial data $\phi (r) = 0$
and $\psi (r) = 100 h (r)$, and values of $\gamma = 1 , 5 , 10$.
The results (depicted in the left column of Figure \ref{Fig4-3-2})
show a loss in the total energy as a function of time.

\subsection*{Internal damping}

Consider first the case when $\gamma$ equals zero.
Figure \ref{Fig4-2-2} shows numerical solutions of
(\ref{paperproblem}) at consecutive times, for initial data $\phi (r) = h (r)$ and 
$\psi (r) = h ^\prime (r) + h (r) / r$, with nonlinear term
$G ^\prime (u) = u ^7$, and values
of $\beta = 0, 0.001, 0.003$. We
observe that small values of $\beta$ produce results similar to
those of the corresponding undamped case. To corroborate this claim,
we appeal once more to the relative differences in the $\ell _{2 ,
\Delta x}$-norm of dissipative solutions with respect to the non-dissipative
one. The results are shown in Table \ref{table4}. The results evidence
the continuity of solutions with respect to the parameter $\beta$
for this particular choice of nonlinearity, providing thus numerical
support in favor of the stability of our method.

\smallskip
We want to establish now the continuity of our method for several 
nonlinear terms at a fixed large time. In order to do it, Figure
\ref{Fig4-2-3-1} shows the numerical solutions of (\ref{paperproblem}) 
at time $t = 0.2$, for the nonlinear terms $G ^\prime (u) = u ^3,
u ^5, u ^7$, for two different sets of initial conditions:
$\phi (r) = 0$ and $\psi (r) = 100 (h)$, and $\phi (r) = h (r)$ and
$\psi (r) = 0$, and values of $\beta = 0 , 0.0001, 0.0002$. The graphs in this figure, together with the analysis
of relative differences in the $\ell _{2 , \Delta x}$-norm supplied
in Table \ref{table5} for the first set of initial conditions, evidence the continuity of the numerical
solution given by our method with respect to the parameter $\beta$
for different nonlinearities.

\smallskip
We now consider the case when $\gamma$ is nonzero. We use $G ^\prime (u) =
0, u ^3, u ^5, u ^7, u ^9$, and $\sinh (5 u) - 5 u$, initial data $\phi (r) =
0$ and $\psi (r) = 100 h (r)$, and time $t = 0.2$. Figure \ref{Fig4-3-1} shows numerical solutions of
(\ref{paperproblem}) for values of $\beta = 0, 0.0005, 0.005$. 
The solutions for smaller nonzero values of $\beta$ are indeed
closer to the corresponding internally undamped solution, while the larger values
of $\beta$ spread out the internally undamped solution at the
same time that the maximum amplitude is decreased.

\smallskip
In order to study the time behavior of the solutions near the origin
we have included Figure \ref{Fig4-31}, using initial data $\phi (r) =
0$ and $\psi (r) = 100 h (r)$, the nonlinear terms $G ^\prime (u) = u
^3$, $u ^5$ and $u ^7$, different choices of values for $\beta$ 
and $\gamma$, and $\Delta r = \Delta t = 0.002$. The left column shows the time-dependence of solutions
at the origin for $\beta = 0$ and three positive values of $\gamma$,
whereas the right column shows similar results for $\gamma = 0$ and
three positive values of $\beta$. We observe that the value of 
solutions at the origin for large times is always
approximately equal to zero for $\beta = 0$, which is in agreement with
our experience of the $(1+1)$-dimensional case.

\smallskip
Finally, Figure \ref{Fig4-3-2} shows the graphs of the energy $E _0$ vs.
time for $G ^\prime (u) = u ^3, u ^5 , u ^7$, initial data $\phi
(r) = 0$ and $\psi (r) = 100 h (r)$, and values of $\beta = 0.0005
, 0.001 , 0.005$. The results (depicted in the right column) show
a loss in the total energy as a function of time. It is clear that the
rate at which the total energy is lost due to internal damping is
greater than the corresponding rate due to external damping.

\section{Discussion}

A numerical method has been designed to approximate radially symmetric
solutions of some dissipative, nonlinear, modified Klein-Gordon
equations with constant internal and external damping coefficients
$\beta$ and $\gamma$, respectively. Our finite-difference scheme is in 
general agreement with the non-dissipative results presented in 
\cite{StraussVazquez}. The method is consistent $\mathcal {O} (
\Delta t ^2) + \mathcal {O} (\Delta r ^2)$, conditionally stable,
and continuous with respect to the 
parameters $\beta$ and $\gamma$; as expected, the total energy decays in time 
for positive choices of the parameters.
The corresponding scheme to approximate the total energy of the system is
consistent and has the property that the discrete rate of change of the
discrete energy with respect to time approximates the corresponding
continuous rate of change for $\beta = 0$. 

\smallskip
Several  conclusions can be drawn from our numerical computations. First of
all, we have seen that both internal and external damping tend to decrease
the magnitude of solutions, as it was expected. Our results clearly exhibit 
the dispersive effects of the parameter $\beta$ and the dissipative effects of $\gamma$. 
Our energy computations evidence the fact that the rate at which the energy is 
dissipated by the internal damping is faster than the corresponding rate 
of external damping. Finally, we observe that the effect of the nonlinear
term in the temporal behavior near the origin is to increase the number of
oscillations as the degree of the nonlinearity is increased. Invariably,
the solutions of the dissipative modified Klein-Gordon equation converge
in time to the trivial solution.


\begin{thebibliography}{10}

\bibitem{Solitons}
{Remoissenet, M.}
\newblock {\em {Waves Called Solitons}}.
\newblock Springer-Verlag, New York, third edition, 1999.

\bibitem{Lomdahl}
{Lomdahl, P. S.; Soerensen, O. H.; Christiansen, P. L.}
\newblock {Soliton Excitations in Josephson Tunnel Junctions}.
\newblock {\em Phys. Rev. B}, {\bf 25}(9):5737--5748, 1982.

\bibitem{Makhankov}
{Makhankov, V. G.; Bishop, A. R.; Holm, D. D.}, editor.
\newblock {\em {Nonlinear Evolution Equations and Dynamical Systems Needs '94;
  Los Alamos, NM, USA 11-18 September '94: 10th International Workshop}}.
\newblock World Scientific Pub. Co. Inc., Singapore, first edition, 1995.

\bibitem{RussianBook}
{Polyanin, A. D.; Zaitsev, V. F.}
\newblock {\em {Handbook of Nonlinear Partial Differential Equations}}.
\newblock Chapman \& Hall CRC Press, Boca Raton, Fla., first edition, 2004.

\bibitem{Glassey}
{Glassey, R. M.}
\newblock {Blow-up theorems for nonlinear wave equations}.
\newblock {\em Math. Zeit.}, {\bf 132}:183--203, 1973.

\bibitem{Jorgens}
{J\"{o}rgens, K}.
\newblock {Das Anfangswertproblem im Grossen f\"{u}r eine Klasse nichtlinearer
  Wellengleichungen}.
\newblock {\em Math. Zeit.}, {\bf 77}:295--308, 1961.

\bibitem{Barone}
{Barone, A.; Esposito, F.; Magee, C. J.; Scott, A. C.}
\newblock {Theory and Applications of the Sine-Gordon Equation}.
\newblock {\em Riv. Nuovo Cim.}, {\bf 1}:227--267, 1971.

\bibitem{StraussVazquez}
{Strauss, W. A.; V\'{a}zquez, L.}
\newblock {Numerical Solution of a Nonlinear Klein-Gordon Equation}.
\newblock {\em J. Comput. Phys.}, {\bf 28}:271--278, 1978.

\bibitem{AlexHabib}
{Alexander, F. J.; Habib, S.}
\newblock {Statistical Mechanics of Kinks in $1 + 1$ Dimensions}.
\newblock {\em Phys. Rev. Lett.}, {\bf 71}:955--958, 1993.

\bibitem{Segal}
{Segal, I. E.}
\newblock {The Global Cauchy problem for a relativistic scalar field with power
  interaction}.
\newblock {\em Bull. Soc. Math. Fr.}, {\bf 91}:129--135, 1963.

\bibitem{Morawetz}
{Morawetz, C. S.; Strauss, W. A.}
\newblock {Decay and scattering of solutions of a nonlinear relativistic wave
  equation}.
\newblock {\em Comm. Pure and Appl. Math.}, {\bf 25}:1--31, 1972.

\bibitem{StraussBook}
{Strauss, W. A.}
\newblock {\em {Partial Differential Equations: An Introduction}}.
\newblock John Wiley \& Sons, Inc., New York, second edition, 1992.

\bibitem{Thomas}
{Thomas, J. W.}
\newblock {\em {Numerical Partial Differential Equations}}.
\newblock Springer-Verlag, New York, first edition, 1995.

\end{thebibliography}
\end{document}